\documentclass[a4paper, 11pt]{article}

\usepackage{longtable}
\usepackage{geometry}
\usepackage{graphicx}
\usepackage{subfig}
\usepackage{color}
\usepackage{xcolor}
\usepackage{bm}
\usepackage{fixmath}
\usepackage{mathrsfs}
\usepackage{ulem}
\usepackage{multirow}
\usepackage{pdflscape}
\usepackage{rotating}
\usepackage{lscape}
\usepackage{pdflscape}
\usepackage{titlesec}
\usepackage{float}
\usepackage{hyperref}
\usepackage{eucal}
\usepackage{amsfonts}
\usepackage{amsthm}

\def\b{\begin{eqnarray}}
\def\e{\end{eqnarray}}

\newtheorem{theorem}{Theorem}
\newtheorem{corollary}{Corollary}[theorem]

\newtheorem{definition}{Definition}

\begin{document}

\begin{center}
{\huge \textbf{On Newton's Rule of Signs}}

\vspace {10mm}
\noindent
{\Large \bf Emil M. Prodanov}
\vskip.8cm
{\it School of Mathematical Sciences, Technological University Dublin,} \\
{\it Park House, Grangegorman, 191 North Circular Road, } \\
{\it Dublin D07 EWV4, Ireland} \\
{\it E-Mail: emil.prodanov@tudublin.ie}
\vskip1cm

\end{center}

\vskip2cm

\noindent
\begin{abstract}
\noindent
Analysing the {\it cubic sectors} of a real polynomial of degree $n$, a modification of Newton's Rule of signs is proposed with which stricter upper bound on the number of real roots can be found. A new necessary condition for reality of the roots of a polynomial is also proposed. Relationship between the quadratic elements of the polynomial is established through its roots and those of its derivatives. Some aspects of polynomial discriminants are also discussed --- the relationship between the discriminants of real polynomials, the discriminants of their derivatives, and the quadratic elements, following a ``discriminant of the discriminant" approach.
\end{abstract}


\newpage

\section{Introduction}

Newton's Rule, introduced without proof by Newton \cite{newton} and proven 182 years later by Sylvester \cite{syl}, allows the determination of upper bounds on the number of positive roots and the number of negative roots of a real polynomial of degree $n$ which, without loss of generality, can be taken to be in binomial form: $p(x) = \sum_{k=n}^0 {n \choose k} a_k x^k$. These bounds are found by considering the double sequence of the simple elements (the coefficients of the polynomial) and the quadratic elements of the polynomial --- the latter obtained by taking the square of each of the coefficients of the polynomial and subtracting from it the product of its immediate neighbours. The quadratic elements are the discriminants (modulo positive multiplicative factors) of all possible quadratic polynomials that can be extracted from the given polynomial by taking, also in binomial form, three consecutive coefficients from it. Permanencies in sign in the sequence of quadratic elements are associated with real roots of the polynomial: if there is a permanence in sign in the corresponding simple elements too, then this indicates a possible negative root, while, if there is a variation in the sign in the corresponding simple elements, then there may be a positive root. This is a generalisation of Descartes' rule of signs \cite{descartes} from 1637 (the Descartes rule was made more precise by Gauss \cite{gauss} in 1876).

Newton states \cite{newton} that {\it if all roots of a polynomial are real, then all quadratic elements of the polynomial are positive}, following the basic idea that if all the roots of a given polynomial are real, then the same is true of the reciprocal polynomial and of every derivative of both of them. Expressing the quadratic elements as inequalities for their positivity (the so called {\it Newton inequalities}), allows the reformulation of the necessary condition for reality of all roots of the polynomial: {\it if the roots of the polynomial are all real, then all Newton inequalities hold}.

Positive quadratic elements however are not always associated with real roots. For example, for the quartic polynomial $x^4 - 2x^3 - 2x^2 + 5x + 10$, all quadratic elements are positive (hence their sequence exhibits only permanencies of sign). But this quartic has two pairs of complex conjugate roots.

Performing the analysis at the level of cubic polynomials (instead of quadratic), that is, considering the {\it cubic sectors of a polynomial}, allows one to get a clearer picture of what type of roots are associated with the two quadratic elements of each of the cubic polynomials. As with the {\it quadratic sectors} of the polynomial, these cubic polynomials are all possible cubics that can be obtained by taking in binomial form four consecutive coefficients of the original polynomial. Negativity of either of the two quadratic elements of a cubic polynomial warrants a pair of complex roots. However, the case of positive quadratic elements needs further investigation. Each quadratic element (except the two at either end of the polynomial, defined as $a_n^2$ and $a_0^2$) has two {\it adjacent coefficients} in the polynomial (simple elements of the polynomial, each belonging to a different cubic sector) --- these are the two coefficients sitting on either side of the three which form the quadratic element. Using a new criterion on the coefficients of the cubic polynomial for the reality of its roots, instead of considering the sign of its discriminant, two types of positive quadratic elements are identified in this work. Real roots of a cubic are associated with a positive quadratic element for which the adjacent coefficient in the cubic lies in a particular finite interval, the endpoints of which are determined by the coefficients of the cubic, forming the quadratic element. However, a cubic with a positive quadratic element with adjacent coefficient not in its relevant interval has a pair of complex conjugate roots. If a positive quadratic element happens to have either of its adjacent coefficients outside of their prescribed relevant intervals (hence, associating with complex roots), such quadratic element, referred to in this work as {\it falsely positive quadratic element}, should, depending on its neighbours in the sequence of quadratic elements, have its sign changed to negative. The Newton Rule thus gets modified and, in result, it will exhibit a maximum number of real roots reduced by an even number (with the minimum number of complex roots increased, of course, by at least a pair).

The analysis of the {\it cubic sectors} of a real polynomial also allows the formulation of a new necessary condition for the reality of its roots: {\it if the roots of a real polynomial are all real numbers, then each quadratic element of the polynomial is positive and each adjacent coefficient lies in its prescribed relevant interval}. An equivalent form of this new necessary condition is the following: {\it if all roots are real, then the polynomial cannot have negative, zero, or falsely positive quadratic elements}.

This work starts with the presentation of a geometrical viewpoint with which the influence of the variation of the free term of a polynomial on the sign of the discriminant is established and, hence, on the possible number of real roots. The quadratic elements of a polynomial are also introduced in this section. This is followed by a section introducing Newton's Incomplete and Complete rules and also establishing a relationship between the quadratic elements through the roots of the polynomial and its derivatives: this allows one to express the ratio of two quadratic elements as ratio of the sums of squares of the differences of these roots or their reciprocals. Section 5 introduces Rosset's \cite{rosset} analysis of the discriminants of the cubic sectors of a polynomial and the introduction of a stricter necessary condition for the reality of all of the roots. The falsely positive quadratic elements are also introduced in this section and their association with cubics with complex roots is demonstrated. Section 6 proposes the modification of Newton's method and the following section offers an in-depth analysis of an example with a polynomial of degree 8, followed by comparison of the modified Newton's method with the original Newton's method, the Descartes' rule of signs, Fourier's theorem, and Sturm's theorem. It is also shown that the modified Newton's method can be improved when considered in conjunction with the Descartes' rule of signs. Section 8 offers a detailed algorithm for the application of the modified Newton's method. Section 9 introduces a new necessary condition for the reality of all roots, based on the modified Newton's method. Another example is given --- a quintic polynomial --- with which this new necessary condition is demonstrated and other aspect of the modified Newton's method are illuminated.
In the final section of this work, an interesting connection between the discriminants of polynomials, the discriminants of their derivatives, and the quadratic elements is also established, following a ``discriminant of the discriminant" approach.

\section{On the Changes of Sign of the Discriminant of a Real Polynomial}

Consider a polynomial of degree $n$ with real coefficients in binomial form
\b
\label{p}
p(x) = \sum_{k=n}^0 {n \choose k} a_k x^k.
\e
(Without loss of generality, any polynomial can be written in binomial form.)

From the Fundamental Theorem of Algebra, a polynomial of degree $n$ has exactly $n$ roots $r_1, \ldots, r_n$. Some of the roots may be repeated. If a complex number is a root, then its complex-conjugate number is also a root. Hence the numbers of complex roots is always even number.

The Vandermonde matrix is the $n \times n$ matrix whose $i^{\mbox{\scriptsize th}}$ row is given by $1, r_i, \ldots, r_i^{n-1}$. Note that each row forms a geometric progression with ratio $r_i$. The determinant of this matrix is called Vandermonde determinant and it is given by $V_n = (-1)^{n(n-1)/2} \prod_{i<j}(r_i - r_j)$. This determinant is also called Vandermonde polynomial.

\begin{definition}
The discriminant $\Delta_n$ of the polynomial $p(x)$ of degree $n$ is defined as $a_n^{2n-2} \prod_{i<j}(r_i - r_j)^2$.
\end{definition}

Hence, the discriminant of a polynomial is $a_n^{2n-2}$ times the square of the determinant $V_n$ of the Vandermonde matrix: $\Delta_n  = a_n^{2n-2} V_n^2$.

The discriminant $\Delta_n$ is also equal, modulo multiplicative factor, to the Sylvester resultant $R(p, p')$ of $p(x)$ and its derivative $p'(x)$. (For example, for the cubic polynomial $a x^3 + b x^2 + c x + d$, the discriminant is $-27 a^2 d^2 + 18 a b c d - 4 a c^3 - 4 b^3 d + b^2 c^2$, while the resultant is equal to the discriminant, multiplied by $-a$.) The Sylvester resultant is the determinant of the $2n \times 2n$ Sylvester matrix $S(p, p')$, where, in general, the $(n+m) \times (n+m) $ Sylvester matrix $S(A, B)$ of two polynomials, $A(x) = \alpha_n x^n + \alpha_{n-1} x^{n-1} + \ldots + \alpha_0$ and $B(x) = \beta_m x^m + \beta_{m-1} x^{m-1} + \ldots + \beta_0$, of degrees $n$ and $m$, respectively, is constructed as follows. The first $m$ rows contain the coefficients $\alpha_n, \alpha_{n-1}, \ldots , \alpha_0$ of the polynomial $A(x)$, followed by zeros and shifted on each subsequent row by $0, 1, \ldots, m-1$ steps. The remaining $n$ rows contain the coefficients $\beta_m, \beta_{m-1}, \ldots , \beta_0$ of the polynomial $B(x)$, followed by zeros and shifted on each subsequent row by $0, 1, \ldots, n-1$ steps. Defining $\alpha_i = 0$, for $i < 0$ and for $i > n$, and defining $\beta_i = 0$, for $i < 0$ and for $i > m$, one immediately sees that for $1 \le i \le m$ the elements of the Sylvester matrix are given by $S_{i, j}(A, B) = \alpha_{n+i-j}$, while for $m+1 \le i \le m+n$ they are given by: $S_{i, j}(A, B) = \beta_{i-j}$. The discriminant of the polynomial $p(x)$ is given by: $\Delta_n = (-1)^{n(n-1)/2} a_n^{-1} R(p, p')$.

The polynomial (\ref{p}) can be viewed as an element of a congruence of polynomials
\b
\mathcal{P} = \left\{ \sum_{k=n}^1 {n \choose k} a_k x^k + \alpha \,\, \vert \,\,  \alpha \in \mathbb{R} \right\},
\e
the graphs of which foliate the $xy$-plane by the continuous variation of the foliation parameter $\alpha$. The elements in this congruence differ from each other by the value of their free term and they all have the same set of stationary points $\{ \mu_i \}_{i=1}^N$ --- the real roots of the equation $p'(x) = n \sum_{k=n}^1 {n - 1 \choose k} a_k x^{k-1} = 0$.  Clearly, $N = n - 1 - 2m$, where $2m$ is the number of complex roots with $m \le (n - 1)/2$ being a non-negative integer.

Within the congruence, there are exactly $N$ (not necessarily all different) {\it ``privileged"} polynomials
\b
p_i(x) = p(x) - p(\mu_i),
\e
for which the abscissa is tangent to their graph at the stationary point $\mu_i$. Each {\it privileged} polynomial $p_i(x)$ has a zero discriminant, as each one of them has at least one root which has multiplicity of 2 or higher.

The variation of the foliation parameter $\alpha$ produces continuous sets of polynomials such that all discriminants within a set have the same sign. Each time the parameter $\alpha$ ``{\it traverses}" the value $a_0 - p(\mu_i)$, a new continuous set of polynomials is obtained and the discriminants in this new set have sign opposite to that of the previous set (unless $\alpha$ ``traverses" an even number of $a_0 - p(\mu_i)$ at the same time; in which case the discriminants of the polynomials in the new set will bounce back from 0).

\begin{theorem} $\phantom{emp}$
\begin{itemize}
\item[(i)] If the discriminant of a polynomial of degree $n \ge 4$ is positive, then the number of complex roots of the polynomial is a multiple of 4.
\item[(ii)] If the discriminant of a polynomial of degree $n \ge 3$ is negative, then there are $2m + 1$ pairs of complex-conjugate roots, where $0 \le m \le (n - 2)/4.$
\item[(iii)] If the discriminant of a polynomial of degree $n \ge 2$ is zero, then there is a repeated real or complex root.
\end{itemize}
\end{theorem}
\begin{proof}
By definition, the discriminant of a polynomial of degree $n$ is given by $\Delta_n = a_n^{2n-2} \prod_{i < j} (x_i - x_j)^2$, where $x_1, \ldots, x_n$ are all $n$ roots (real and complex) of the polynomial. If all roots are real, then the differences $x_i - x_j$ for $i < j$ are all real and their squares are positive. Hence, $\Delta_n > 0$ in this case. If there is just one pair of complex conjugate roots (and, thus, $n-2$ real roots), then the difference of the complex roots will be purely imaginary and its square will be negative. The difference of a real root and one of the pair of complex conjugate roots will be a number, complex conjugate to the difference of the same real root and the other complex root. The product of these two complex conjugate differences is a real number and so its square is positive. The difference of two real roots is real and its square is positive. Therefore, the determinant of the polynomial with one pair of complex conjugate roots is negative. If there are $k$ pairs of complex conjugate roots, the difference of the complex conjugate roots within each pair of such will be purely imaginary; all other differences will be either real or will appear as complex conjugate pairs. Hence, the determinant of the polynomial will be positive if $k$ is even and negative if $k$ is odd (as it will be a product involving $k$ negative factors. One can invert this argument: if the discriminant is negative, then there is an odd number $k = 2m + 1$ pairs of complex-conjugate roots. Clearly, the number of complex roots is $2k = 4m + 2$ and this cannot exceed $n$, hence $0 \le m \le (n-2)/4$.

\begin{center}
\begin{tabular}{cc}
\includegraphics[width=68mm, height=58mm]{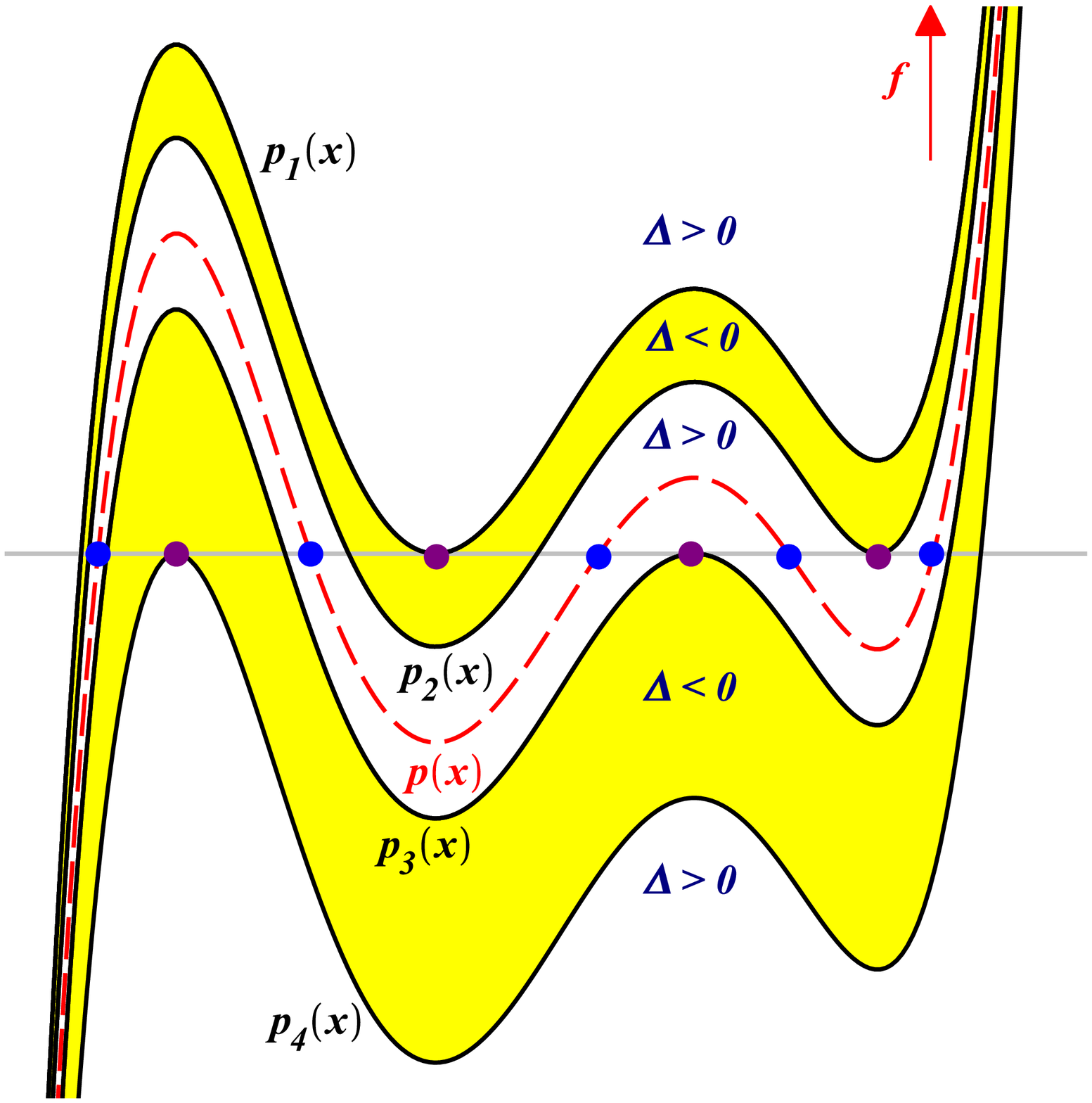} & \includegraphics[width=68mm, height=58mm]{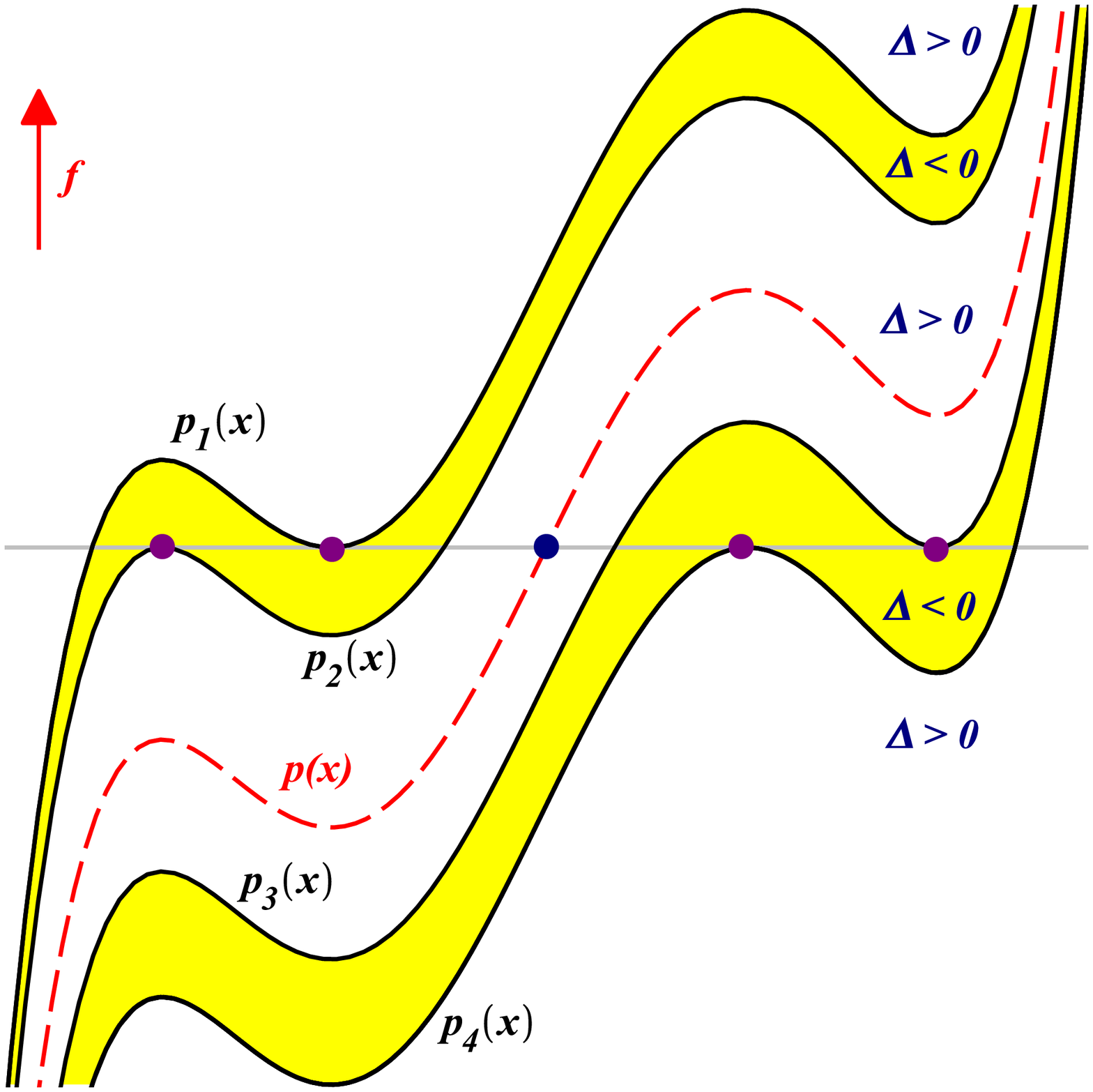}  \\
{\scriptsize {\bf Figure 1a}} &  {\scriptsize {\bf Figure 1b}} \\
\multicolumn{1}{c}
{\begin{minipage}{18em}
\scriptsize
\vskip-2.00cm
The quintic polynomial $x^5 -3 x^4 - x^3 + 7 x^2 - (3/2) x + f$ with varying free term $f$. The four {\it privileged} polynomials $p_i(x)$ have a double root each. The free terms of these polynomials are $0.082, \,\, -0.944, \,\, -2.837,$ and $-5.530$, respectively. When $f$ is sufficiently large, that is, for $f > 0.082$, the polynomial is ``above" $p_1(x)$. The discriminant is positive and the polynomial has one real root and two pairs of complex conjugate roots. In the polynomial band bounded by the privileged polynomials $p_1(x)$ and $p_2(x)$, that is, for $-0.944 < f < 0.082$, all polynomials have a negative discriminant and, hence, 3 real roots and a pair of complex conjugate roots. In the next polynomial band, bounded by $p_2(x)$ and $p_3(x)$, that is, $-2.837 < f < -0.944$, all polynomials have positive discriminant and five real roots. With the decrease of the foliation parameter $f$, the situation starts to ``unwind": the polynomial band between $p_3(x)$ and $p_4(x)$, namely, for $-5.530 < f < -2.837$, is characterised by negative discriminants and, thus, 3 real roots and one pair of complex conjugate roots. Finally, ``below" $p_4(x)$, that is, for that is, for $f < -5.530$, all polynomials have positive discriminants and 1 real root and two pairs of complex conjugate roots.
\end{minipage}}
&
\multicolumn{1}{c}
{\begin{minipage}{18em}
\scriptsize
Another quintic polynomial, $5 x^5 + (1/10) x^4 - 8 x^3 - (1/4) x^2 + 4 x + f$, with varying free term $f$ and, again, with four {\it privileged} polynomials $p_i(x)$ having a double root each. The free terms of these polynomials are $1.215, \,\, 0.834, \,\, -0.572,$ and $-1.117$, respectively. As the free term $f$ is varied from $+\infty$ to $-\infty$, the number of real roots of the resulting polynomials does not necessarily increase monotonically from one polynomial band to another until only real roots are encountered (as in Figure 1a). In fact, the number of real roots may increase or decrease as one ``shifts through" the different polynomial bands. This polynomial cannot have five real roots. When $f$ is sufficiently large, that is, for $f > 1.215$, the discriminant is positive and the polynomial has one real root and two pairs of complex conjugate roots. In the polynomial band bounded by  $p_1(x)$ and $p_2(x)$, that is, for $0.834 < f < 1.215$, all polynomials have a negative discriminant and, hence, 3 real roots and a pair of complex conjugate roots. However, in the next polynomial band, bounded by $p_2(x)$ and $p_3(x)$, that is, for $-0.572 < f < 0.834$, all polynomials have positive discriminant and only one real root and two pairs of complex-conjugate roots. The polynomial band between $p_3(x)$ and $p_4(x)$, namely, for $-1.117 < f < -0.572$, is characterised by negative discriminants and, thus, there are 3 real roots and one pair of complex conjugate roots Finally, ``below" $p_4(x)$, that is, for that is, for $f < -1.117$, all polynomials have positive discriminants and 1 real root and two pairs of complex conjugate roots.
\end{minipage}} \\
& \\
\end{tabular}
\end{center}

On the other hand, a polynomial with no complex roots has a positive discriminant. Every two pairs of complex conjugate roots are also associated with a positive discriminant. Hence, if the discriminant of a polynomial is positive, then the number of complex roots of the polynomial of degree 4 or higher is a multiple of 4, namely, it is $4m$, with $m \ge 0$.

Finally, if there is a repeated root (regardless of whether it is real or complex), the discriminant will be zero as at least one of the factors $x_i - x_j$ will be zero.
\end{proof}

In this manner, different {\it ``polynomial bands"} are formed in the $xy$-plane and within each of these polynomial bands, the discriminants of all polynomials have the same sign. These bands are bounded by the {\it privileged} polynomials --- see Figure 1a and Figure 1b.

If a polynomial has only real roots, then all of its derivatives have only real roots \cite{newton}. The converse is not true: if all derivatives of a polynomial have only real roots, it does not follow that the roots of the polynomial are all real \cite{newton}. For the example with the quartic polynomial $x^4 - 2x^3 - 2x^2 + 5x + 10$, which has four complex roots (referred to in the Introduction), all of the derivatives of the polynomial have real roots only. Hence, all of the derivatives of a polynomial having only real roots is a necessary condition for the polynomial to have only real roots. If any of the derivatives of a polynomial has at least one pair of complex roots, then the polynomial cannot have only real roots. This means that the number of stationary points of the polynomial will be strictly less than $n - 1$. Hence, the number of {\it privileged} polynomials $p_i(x)$ will be strictly less than $n - 1$. Therefore, the number of different polynomial bands and, hence, the number of continuous sets of polynomials with discriminants of the same sign will be strictly less than $n - 1$. In result, under the variation of the foliation parameter $\alpha$, there will not be enough polynomial bands in the congruence so that the polynomials could alternate the signs of their discriminants from one polynomial band into the other until a positive discriminant and possibility of real roots only is revealed: an even number of polynomial bands is eliminated by the reduction by an even number of {\it privileged} polynomials, which stems, in turn, from the reduction by an even number of the stationary points, due to some derivative(s) having complex roots (whose number is always even). On the contrary, if a polynomial has only real roots, this could only happen in the inner-most polynomial band of $n$ such bands --- for which the discriminant is positive.

Note that, as the free term is varied from $-\infty$ to $+\infty$, the number of real roots of the resulting polynomials does not necessarily increase monotonically from one polynomial band to another (until a situation with real roots only is reached --- as in Figure 1a). In fact, the number of real roots may increase or decrease as one ``shifts through" the different polynomial bands (Figure 1b).

In order for the number of polynomial bands to be reduced (always by an even number), it would be sufficient if $p'(x)$ has at least one pair of complex conjugate roots. Again, note that this is not a necessary and sufficient condition: if $p'(x)$ has $n - 1$ real roots, all polynomial bands will be present, but the polynomial $p(x)$ should not necessarily have $n$ real roots (as in Figure 1b). In order that $p'(x)$ does not to have $n -1$ real roots, it would be sufficient if $p''(x)$ does not have $n - 2$ real roots. Continuing recursively, one arrives at the bottom rung, $(d^{n-2} / dx^{n-2}) p(x) = (n!/2!) (a_n \,x^2 + 2a_{n-1} \, x + a_{n-2})$, and having no real roots at this bottom rung is a sufficient condition for the polynomial $p(x)$ to have at least a pair of complex-conjugate roots. Thus, if the discriminant of the bottom-rung quadratic polynomial, or the so called {\it quadratic element},
\b
A_{n-1} = a_{n-1}^2 - a_n \, a_{n-2}
\e
is negative, it will guarantee the elimination of at least two of the polynomial bands and, hence, guarantee the impossibility for the polynomial to have real roots only.

\section{The Incomplete and the Complete Newton's Rules}

As one will be interested in the number of real roots of the polynomial, if not all powers of the variable $x$ are present in $p(x)$, the missing terms can be ``restored" in a very simple manner --- through a coordinate transformation of the type $x \to x + \beta$ (which may need to be repeated). This works in direction, opposite to that of depressing a polynomial. The new polynomial $p(x + \beta)$ differs from the original polynomial $p(x)$ only in that its graph is translated $\beta$ units to the left (if $\beta$ is positive) or $\beta$ units to the right (if $\beta$ is negative): both polynomials have the same number of real roots.

To illustrate this, consider the coefficients of the polynomial $p(x)$. Starting from $a_n$ and moving towards $a_0$, let $a_m$ (with $m \le n$) be the last non-zero coefficient, that is, $a_{m-1} = 0$ (the term $a_{m-1} x^{m-1}$) is not present in the polynomial. Hence, one has $p(x) = {n \choose n} a_n x^n + \ldots + {n \choose m} a_m x^m + {n \choose m-2} a_{m-2} x^{m-2}$ $ + \ldots + {n \choose 0} a_0 x$. The missing term $a_{m-1} x^{m-1}$ can be introduced, with arbitrary coefficient $a_{m-1}$, through the coordinate transformation: $x \to x + (a_{m-1}/a_m)/(n - m + 1)$. Therefore, the number of real roots of the original polynomial $ \ldots + {n \choose m} a_m x^m + {n \choose m-2} a_{m-2} x^{m-2} + \ldots$ will be the same as the number of real roots of the new  polynomial $\ldots + {n \choose m} a_m x^m + {n \choose m-1} a_{m-1} x^{m-1} + {n \choose m-2} (1/a_m) \{(n- m - 2)/[2(n-m+1)] a_{m-1}^2 + a_{m-2} a_m \} x^{m-2} + \ldots$ Although the polynomial obtained in this way is different from the original, their graphs exhibit translational invariance and the two polynomials have the same number of real roots.

{\it Without loss of generality, for the purpose of real root counting, one can assume that none of the coefficients of the polynomial $p(x)$ is zero}.

{\it This argument can also be extended to ensure that the quadratic elements are also all nonzero}.

Having already encountered the quadratic element $A_{n-1} = a_{n-1}^2 - a_n \, a_{n-2}$, consider the polynomial $p^\dagger(x)$, reciprocal to $p(x)$:
\b
\label{pd}
p^\dagger(x) = \sum\limits_{k=0}^n {n \choose k} a_k x^{n-k} = x^n \, p(\frac{1}{x}).
\e
The coefficients of the reciprocal polynomial $p^\dagger(x)$ are the coefficients of $p(x)$ in reverse order. The roots of two reciprocal polynomials are reciprocal, that is, $r$ is a root of $p(x)$ if, and only if, $1/r$ is a root of $p^\dagger(x)$. Hence, the polynomials  $p(x)$ and $p^\dagger(x)$ have the same number of real roots. Applying to $p^\dagger(x)$ the arguments applied earlier to $p(x)$, one arrives at the {\it quadratic element}
\b
A_{1} = a_{1}^2 - a_2 \, a_{0}.
\e
If it is negative, it is then sufficient for $p^\dagger(x)$, and hence for $p(x)$, to have at least a pair of complex conjugate roots.

Applying these arguments to the reciprocal polynomial of the first derivative of $p^\dagger(x)$, that is, to $n \sum_{k=n-1}^0 {n - 1 \choose k} a_k x^k$, yields the {\it quadratic element}
\b
A_{n-2} = a_{n-2}^2 - a_{n-1} \, a_{n-3}.
\e
If it is negative, the polynomial $p(x)$ cannot have $n$ real roots.

From the polynomial, reciprocal to $n \sum_{k=n-1}^0 {n - 1 \choose k} a_k x^k$, one gets the {\it quadratic element}
\b
A_{2} = a_{2}^2 - a_3 \, a_{1}.
\e
If it is negative, the polynomial $p(x)$ cannot have $n$ real roots.

Continuing in this vein, one can ``{\it resolve}" the polynomial $p(x)$ into a collection of {\it quadratic elements}
\b
A_{m-1} = a_{m-1}^2 - a_m \, a_{m-2},
\e
where $2 \le m \le n$ (see also \cite{wagner, nic, acosta}).

\begin{definition}
The quadratic polynomials which can be obtained from the original polynomial $p(x)$ by taking in binomial form three consecutive coefficients from it are called \linebreak \underline{quadratic sectors} of the polynomial.
\end{definition}

The {\it quadratic elements} are thus the discriminants (modulo positive multiplicative factors) of all quadratic sectors of the polynomial $p(x)$.

Newton states  the following on page 365 of \cite{newton}: ``{\it Although it is a certain Criterion, that there are two impossible Roots as often as the Square of any Term (...) is deficient of the product of the Terms adjacent, yet it is no Proof that the Roots are real if the Square of any Term (...) exceeds the product of the adjacent Terms;...}" and, ibidem, ``{\it Lastly, every Rule, depending upon the Comparison of the Square of a Term with the Product of the adjacent Terms on either Side, must sometimes fail to discover the impossible Roots, because the Number of such Comparisons being always less by Unity than the Number of Quantities in the Equation,...}"

Hence, one has the following

\begin{theorem}[Newton, see also Theorem 4 in \cite{wagner}]
A sufficient condition for the existence of complex roots is the occurrence of a non-positive quadratic element.
\end{theorem}

Newton did not prove his Rule and it remained unproven for 182 years --- until 1865 when Sylvester \cite{syl} proved it.

Suppose that the quadratic elements are all non-zero (if some are, then a coordinate translation of the type considered earlier, would cure this; vanishing quadratic elements were not considered by Newton, but such trick was alluded to by him \cite{newton}, pages 373 and further, and used by Sylvester in his proof \cite{syl} of Newton's Rule).

Introduce next the quadratic elements $A_0 = a_0^2$ and $A_n = a_n^2$, as the two end-elements of the sequence of the {\it quadratic elements} $A_i$, $i = 1, \ldots, n-1$, and consider their signs. When presenting his proof of Newton's rule, Sylvester states \cite{syl}: ``{\it By a group of negative signs, or a negative group, if we understand a sequence of negative signs, with no positive sign intervening, this incomplete rule may be stated otherwise, as follows: The number of imaginary roots of an algebraic function cannot be less than the number of negative groups in the complete series of its quadratic elements}". This is referred to as {\it Newton's Incomplete Rule}. As the two end quadratic elements are both positive, this lower bound is always an even number.

\begin{definition}
The coefficients $a_m$ ($0 \le m \le n$) of the polynomial are called \underline{simple} \underline{elements}.
\end{definition}

Consider the natural sequence of the simple elements, together with the natural sequence of the quadratic elements, written underneath them in such way that $a_m$ is directly above $A_m$, forming a double sequence. Sylvester \cite{syl} gives the following

\begin{definition}
In the double sequence of simple and quadratic elements,  \raisebox{-.1cm}{$\stackrel{\scriptstyle{a_m}}{\scriptstyle{A_m}}$} is called an \underline{adjacent coefficient couple of elements} \,\, and \raisebox{-.1cm}{$\,\,\, \stackrel{\scriptstyle{a_m \,\,\, a_{m+1}}}{\scriptstyle{A_m \,\,\, A_{m+1}}}$} is called an \underline{adjacent coefficient} \underline{couple of successions}.
\end{definition}

Using Sylvester's terminology \cite{syl}, for each adjacent coefficient couple of successions, there are four possibilities:

\begin{definition} $\phantom{emp}$
\begin{itemize}
\item [(i)] \underline{Double Permanence} (denoted by $pP$) --- when the simple elements are of the same sign and the quadratic elements are also of the same sign;
\item [(ii)] \underline{Double Variation} (denoted by $vV$) --- when the simple elements have opposite sign and the  quadratic elements have opposite sign;
\item [(iii)] \underline{Permanence Variation} (denoted by $pV$)--- when the simple elements have the same sign and the quadratic elements have opposite sign;
\item [(iv)] \underline{Variation Permanence} (denoted by $vP$)--- when the simple elements have opposite sign, but the quadratic elements have the same sign.
\end{itemize}
\end{definition}

The {\it Complete Newton's Rule}, stated by Newton \cite{newton} and proved by Sylvester \cite{syl}, is:

\begin{theorem}[Newton]$\phantom{emp}$
\begin{itemize}
\item [(1)] {\it The number of negative roots is less than or equal to the number of double permanencies ($pP$)}.
\item [(2)] {\it The number of positive roots is less than or equal to the number of variation permanencies ($vP$)}.
\end{itemize}
\end{theorem}

Two immediate corollaries follow \cite{syl}:

\begin{corollary}[Sylvester]
The number of real roots is less than or equal to the number of double permanencies plus the number of variation permanencies, or, simply, the number of permanencies in the sequence of quadratic elements.
\end{corollary}

\begin{corollary}[Sylvester]
The number of complex roots is greater than or equal to the number of variations in the sequence of quadratic elements.
\end{corollary}

Newton's claim that ``{\it it is no Proof that the Roots are real if the Square of any Term (...) exceeds the product of the adjacent Terms}" \cite{newton}, presents a necessary condition for the reality of all roots of a polynomial. Namely,
{\it if all roots of a polynomial are real, then all Newton inequalities hold, that is:}
\b
A_m = a_m^2 - a_{m-1} \, a_{m+1} > 0
\e
for all $m = 1, 2, \ldots, n-1$. The converse, as already discussed, is not true. See also Theorem 3 in \cite{wagner}. Note that $A_0 = a_0^2 > 0$ and $A_n = a_n^2 > 0$.

\section{Relationship between the Quadratic Elements through the Roots of the Real Polynomial and Its Derivatives}

Let $\lambda_1, \lambda_2, \ldots ,\lambda_n$ be the zeros of the polynomial (\ref{p}); $\mu_1, \mu_2, \ldots , \mu_{n-1}$ be the zeros of its first derivative; $\nu_1, \nu_2, \ldots ,\nu_{n-2}$ --- those of its second derivative and so forth until $\theta_1, \theta_2$ --- the zeros of its derivative of order $n-2$.

Using the relationship between the zeros of a polynomial and those of its derivative, see corollary (2.2) in Whitely \cite{white} and, also, \cite{monov}, one gets:
\b
\label{wh}
\sum_{i > j} (\lambda_i - \lambda_j)^2 & = & \frac{n^2}{(n-1)(n-2)} \sum_{i > j} (\mu_i - \mu_j)^2 \nonumber \\
& = & \frac{n^2}{(n-1)(n-2)} \, \frac{(n-1)^2}{(n-2)(n-3)}
\sum_{i > j} (\nu_i - \nu_j)^2 \nonumber \\
& = & \ldots \,\,\,\, =  \,\, \frac{1}{4} n^2 (n-1) (\theta_2 - \theta_1)^2
\nonumber \\ & = & \,\,
\frac{1}{a_n^2} \, n^2 (n-1) (a_{n-1}^2 - a_n \, a_{n-2}) = \frac{1}{a_n^2} \,n^2 (n -1) A_{n-1},
\e
given that the derivative of order $(n-2)$ of the polynomial (\ref{p}) is $(n!/2!) (a_n x^2 + 2 a_{n-1} x + a_{n - 2})$ and, hence, $(\theta_1 - \theta_2)^2$ is $1/a_n^2$ times its discriminant, that is $(\theta_1 - \theta_2)^2 = (4/a_n^2) \, (a_{n-1}^2 - a_n \, a_{n-2}) = (4/a_n^2) A_{n-1}$.

Therefore, the quadratic element $A_{n-1} = a_{n-1}^2 - a_n \, a_{n-2}$ is equal to (modulo positive multiplicative factors) the sums of the squares of the differences of the roots of all derivatives of the polynomial from order 0 (the polynomial itself) to order $n-2$ --- the quadratic polynomial $(n!/2!) (a_n x^2 + 2 a_{n-1} x + a_{n - 2})$; all these sums have the same sign
as the quadratic element $A_{n-1}$:
\b
\label{aa}
A_{n-1} & = & \frac{a_n^2}{n^2 (n-1)} \sum_{i > j} (\lambda_i - \lambda_j)^2 \,\, = \,\, \frac{a_n^2}{(n-1)^2 (n-2)} \sum_{i > j} (\mu_i - \mu_j)^2
\nonumber \\
& = & \frac{a_n^2}{(n-2)^2 (n-3)} \sum_{i > j} (\nu_i - \nu_j)^2 = \,\, \ldots \,\, = \frac{a_n^2}{4} \, (\theta_2 - \theta_1)^2.
\e

If all of the roots of the polynomial are real, then all these sums will be positive, including $A_{n-1}$. On the other hand, positivity of $A_{n-1}$ and, hence, positivity of all these sums, does not necessarily mean that the roots of the polynomial are real: should one or more pair of complex roots be present, then the sum of the squares of the differences of the roots could be negative or positive. However, if in the sequence (\ref{aa}), one of the sums is negative, then all others are negative too, together with $A_{n-1}$. And vice versa. For example, if $A_{n-1}$ is positive, then the quadratic polynomial $(d^{n-2}/dx^{n-2}) \, p(x)$ has two real roots $\theta_{1,2}$. Hence, the roots of the cubic polynomial $(d^{n-3}/dx^{n-3})  \,\, p(x)$ may be all real or there may be a complex pair, however, the sum of the squares of their differences will also be positive. This applies to all earlier derivatives of the polynomial.

If one considers again the reciprocal polynomial (\ref{pd}), following the same arguments, a similar relationship can be established between $A_{m}$ and the sums of the squares of the differences of the roots of all derivatives of the reciprocal polynomial of order $m-2$ to order $0$. The roots of the reciprocal polynomial are $1/\lambda_i$ ($i = 1, 2, \ldots n$). Let $q_i \,\, (i = 1, 2, \ldots , n-1), \,\,$ $r_i \,\, (i = 1, 2, \ldots , n-2), \,\, \ldots \, , s_{1,2}$ be the roots of its first, second, $\ldots, \,\, (n-2)^{\mbox{\scriptsize th}}$ derivative, respectively. In a similar manner, one gets:
\b
\label{aaa}
A_{1} & = & \frac{a_0^2}{n^2 (n-1)} \sum_{i > j} \left( \frac{1}{\lambda_i} - \frac{1}{\lambda_j} \right)^2 \,\, = \,\, \frac{a_0^2}{(n-1)^2 (n-2)} \sum_{i > j} (q_i - q_j)^2
\nonumber \\
& = & \frac{a_0^2}{(n-2)^2 (n-3)} \sum_{i > j} (r_i - r_j)^2 = \,\, \ldots \,\, = \frac{a_0^2}{4} \, (s_2 - s_1)^2.
\e
Using (\ref{aa}) and (\ref{aaa}), one can relate the two quadratic elements through the roots of the polynomial or the roots of its derivatives:
\b
\frac{A_1}{A_{n-1}} \,\,\, = \,\,\, \frac{a_0^2}{a_n^2} \,\,\, \frac{\sum\limits_{i > j} \left(\frac{1}{\lambda_i} - \frac{1}{\lambda_j} \right)^2}{\sum\limits_{i > j} (\lambda_i - \lambda_j)^2}.
\e

Likewise, considering $p'_n(x)$ and its reciprocal polynomial (whose roots are $1/\mu_i, \,\, i = 1, 2, \ldots, n-1$), one can relate $A_2$ with $A_{n-1}$:
\b
\frac{A_2}{A_{n-1}} \,\,\, = \,\,\, \frac{a_1^2}{a_n^2} \,\,\, \frac{\sum\limits_{i > j} \left(\frac{1}{\mu_i} - \frac{1}{\mu_j} \right)^2}{\sum\limits_{i > j} (\mu_i - \mu_j)^2}.
\e
From $p''_n(x)$ and its reciprocal polynomial (whose roots are $1/\nu_i \,\, i = 1, 2, \ldots, n-2$), it is straightforward to obtain another relation:
\b
\frac{A_3}{A_{n-1}} \,\,\, = \,\,\, \frac{a_2^2}{a_n^2} \,\,\, \frac{\sum\limits_{i > j} \left(\frac{1}{\nu_i} - \frac{1}{\nu_j} \right)^2}{\sum\limits_{i > j} (\nu_i - \nu_j)^2}
\e
and so forth. Hence:
\b
\frac{A_1}{A_2} =  \frac{a_0^2}{a_1^2} \, \frac{\sum\limits_{i > j} \left(\frac{1}{\lambda_i} - \frac{1}{\lambda_j} \right)^2}{\sum\limits_{i > j} \left(\frac{1}{\mu_i} - \frac{1}{\mu_j} \right)^2}, \quad
\frac{A_1}{A_3} = \frac{a_0^2}{a_2^2} \, \frac{\sum\limits_{i > j} \left(\frac{1}{\lambda_i} - \frac{1}{\lambda_j} \right)^2}{\sum\limits_{i > j} \left(\frac{1}{\nu_i} - \frac{1}{\nu_j} \right)^2}, \quad
\frac{A_2}{A_3} = \frac{a_1^2}{a_2^2} \, \frac{\sum\limits_{i > j} \left(\frac{1}{\mu_i} - \frac{1}{\mu_j} \right)^2}{\sum\limits_{i > j} \left(\frac{1}{\nu_i} - \frac{1}{\nu_j} \right)^2}, \ldots
\e

\section{The Cubic Sectors of a Real Polynomial and the Falsely Positive Quadratic Elements}

\subsection{Rosset's Analysis of the Discriminants of Cubic Sectors of a Polynomial}

Instead of considering Newton's {\it quadratic inequalities}, Rosset \cite{rosset} introduces a set of {\it cubic inequalities} in the following manner. Let $x_1, x_2, \ldots, x_n$, with $n \ge 3$, be real numbers; $e_1 = \sum_{i=1} ^n x_i, \,\, e_2 = \sum_{i < j} x_i x_j, \,\, \ldots, e_n = \prod_{i=1}^n x_i$ be the elementary symmetric functions; and $E_m = {n \choose m}^{-1} e_m$ --- the normalized elementary symmetric functions. The polynomial (\ref{p}) in this notation is given by \cite{rosset}:
\b
p(x) = \sum_{k = 0}^n (-1)^k {n \choose k} E_k(x_1, \ldots, x_n) x^{n-k}.
\e
The {\it cubic inequalities} proposed by Rosset are given and proven in:
\begin{theorem}[Rosset \cite{rosset}]
If $n \ge 3$ and $x_1, \ldots , x_n$ are real numbers then for $m = 0, \ldots , n-3$
\b
\label{ross}
6 E_m E_{m+1} E_{m+2} E_{m+3} - 4 E_m E_{m+2}^3 - E_m^2 E_{m+3}^2 - 4 E_{m+1}^3 E_{m+3} + 3 E_{m+1}^2 E_{m+2}^2 \ge 0,
\e
\end{theorem}

The left-hand sides of (\ref{ross}) are nothing else but ($1/27$ of) the discriminants of every cubic polynomial
\b
\label{ro}
E_m x^3 - 3 E_{m + 1} x^2 + 3 E_{m + 2} x - E_{m + 3}
\e
 whose coefficients are four consecutive coefficients (in binomial form) of the original polynomial.

For the purposes of this work, introduced is the following
\begin{definition}
The cubic polynomials which can be obtained from the original polynomial $p(x)$ by taking in binomial form four consecutive coefficients from it are called \linebreak \underline{cubic sectors} of the polynomial.
\end{definition}

As will be shown further in this section (in the final sub-section), each of these cubic polynomials can be obtained from the original polynomial by a suitable sequence of differentiations and passages to reciprocal polynomials (and, in relation to this, one should note that, as the polynomial (\ref{p}) is in binomial form, all of its derivatives, alongside all of their reciprocal polynomials, are also in binomial form --- as is (\ref{ro}), --- modulo overall multiplicative factors).

Contrasting Rosset's set-up \cite{rosset} with Newton's \cite{newton}, one immediately sees that the necessary condition for reality of all of the roots $x_1, \ldots , x_n$ of $p(x)$ in the former set-up is having non-negative discriminants of all cubic sectors of the polynomial, while the necessary condition for reality of the roots in the latter is having non-negative discriminants of all {\it quadratic sectors} of the original polynomial. However, Rosset points out \cite{rosset} that, if $E_0 = 1, E_1, \ldots , E_n$ are real numbers satisfying the cubic inequalities (\ref{ross}) for $m = 0,1, \ldots, n - 3$, then they also satisfy Newton's quadratic inequalities $E_m^2 > E_{m-1} E_{m+1}$, with $m = 1, 2, \ldots, n-1$, while the converse is not true and, hence, Rosset argues \cite{rosset} that the cubic inequalities are stronger than Newton's quadratic inequalities.

On the other hand, Rosset does not offer a tool for the determination of an upper bound of the number of possible real roots of a polynomial: Rosset's cubic inequalities are only used in the context of a necessary condition for the reality of all of the roots of the polynomial (and this necessary condition is more restrictive than Newton's).

\subsection{The Cubic Polynomial and its Quadratic Elements --- New Criterion for the Existence of Three Real Roots}
In this work, the cubic sectors of the polynomial will again be studied, but with the goal of proposing a {\it modification of Newton's Rules of Signs}, allowing the determination of a {\it stricter} upper bound on the number of possible real roots, --- as one would expect that, performing the analysis on the level of the cubic sectors of a polynomial, instead of the quadratic ones, gives a clearer picture of what type of roots are associated with the {\it two neighbouring quadratic elements} of the cubic polynomial.

The quadratic elements of the general real cubic polynomial with coefficients in binomial form,
\b
\label{c}
c(x) = a_3 x^3 + 3 a_2 x^2 + 3 a_1 x + a_0,
\e
are $A_0 = a_0^2, \,\, A_1 = a_1^2 - a_0 a_2, \,\, A_2 = a_2^2 - a_1 a_3$ and $A_3 = a_3^2$. For the purposes of this work, the quadratic elements $A_0$ and $A_3$ of this cubic polynomial are ignored. Clearly, the remaining $A_1$ and $A_2$ involve three consecutive of the four coefficients of the cubic.
\begin{definition}
For a given cubic sector, the \underline{adjacent coefficient} of a quadratic element is this coefficient of the cubic which is not involved in the determination of the corresponding quadratic element.
\end{definition}

Thus, for the cubic (\ref{c}), the {\it adjacent coefficient} of the quadratic element $A_1 = a_1^2 - a_0 a_2$ is $a_3$ and the {\it adjacent coefficient} of the quadratic element $A_2 = a_2^2 - a_1 a_3$ is $a_0$.

\begin{theorem}[see also \cite{27, 29}]
The cubic polynomial $c(x) = a_3 x^3 + 3 a_2 x^2 + 3 a_1 x + a_0 $ has three real roots if, and only if, the quadratic element $A_2 = a_2^2 - a_1 a_3$ is non-negative and the adjacent coefficient $a_0$ lies in the interval $[a^{(0)}_2, a^{(0)}_1]$, with $a^{(0)}_{1,2}$ given by
\b
\label{c12}
a^{(0)}_{1,2} = \frac{- 3 a_2 A_2 + a_2^3 \pm 2 A_2^{3/2}}{a_3^2}.
\e
\end{theorem}
\begin{proof}
The discriminant of this cubic polynomial is
\b
\Delta_3 = 27 [- a_3^2 a_0^2 + 2 a_2 (3 a_1 a_3 - 2 a_2^2) a_0 + 3 a_1^2 a_2^2 - 4 a_1^3 a_3].
\e
It is quadratic in the free term $a_0$ and the discriminant of this quadratic is
\b
\Delta_2 = 16 (a_2^2 - a_1 a_3)^3 = 16 A_2^3.
\e

If the quadratic element $A_2$ is negative, one has $\Delta_2 < 0$. Hence, $\Delta_3 < 0$ for all $a_0$ and therefore the cubic polynomial $c(x)$ with $A_2 < 0$ has only one real root (and a pair of complex conjugate roots).

If the quadratic element $A_2$ is non-negative, then the cubic polynomial $c(x)$ has three real roots, provided that $a_0 \in [a^{(0)}_2, a^{(0)}_1]$, where $a^{(0)}_{1,2}$ are the roots of the quadratic in $a_0$ equation $\Delta_3 = 0$, that is, they are the roots of
\b
\label{quadratic}
a_3^2 a_0^2 - 2 a_2 (3 a_1 a_3 - 2 a_2^2) a_0 - 3 a_1^2 a_2^2 + 4 a_1^3 a_3 = 0.
\e
Namely, the ones given by (\ref{c12}).
\end{proof}

If $a_0 \notin [a^{(0)}_2, a^{(0)}_1]$, then the cubic polynomial $c(x)$ will have one real root and a pair of complex conjugate roots.

Hence, one has a {\it new criterion for the reality of the roots of the cubic polynomial}. It is a condition on the coefficients of the polynomial, rather than the sign of the discriminant: {\it the cubic polynomial $c(x) = a_3 x^3 + 3 a_2 x^2 + 3 a_1 x + a_0 $ has three real roots if, and only if, the quadratic element $A_2 = a_2^2 - a_1 a_3$ is non-negative and the adjacent coefficient $a_0$ lies in the interval $[a^{(0)}_2, a^{(0)}_1]$, with $a^{(0)}_{1,2}$ given by} (\ref{c12}).

This new criterion for the reality of the roots of the cubic is, of course, equivalent to the discriminant $\Delta_3$ of the cubic being non-negative if, and only if, three real roots exist.

\begin{definition}
The interval determined by the roots (\ref{c12}) of (\ref{quadratic}) is called \underline{prescribed} \underline{interval} for the adjacent coefficient $a_0$.
\end{definition}

As already mentioned, the cubic sector $c(x)$ has two quadratic elements: $A_1$ with adjacent coefficient $a_3$ and $A_2$ with adjacent coefficient $a_0$ (recall that the quadratic elements at both ends of the cubic, that is, $A_0$ and $A_3$, are ignored). Hence, there are two adjacent coefficients for each cubic sector and each of these adjacent coefficients has its own prescribed interval.

If one considers instead the reciprocal cubic polynomial $c^\dagger(x) = x^3 c(1/x) = a_0 x^3 + 3 a_1 x^2 + 3 a_2 x + a_3$ and applies the same arguments to it, three real roots will exist if, and only if, the quadratic element $A_1 = a_1^2 - a_0 a_2$ is non-negative and its adjacent coefficient $a_3$ lies in its prescribed interval $[a^{(3)}_2, a^{(3)}_1]$, with $a^{(3)}_{1,2}$ given by
\b
a^{(3)}_{1,2} = \frac{- 3 a_1 A_1 + a_1^3 \pm 2 A_1^{3/2}}{a_0^2}.
\e

If $c(x)$ has three real roots, $c^\dagger(x)$ will also have three real roots (and vice versa) --- as the roots of reciprocal polynomial are themselves reciprocal. Hence, $A_1 \ge 0$ and $a_3 \in [a^{(3)}_2, a^{(3)}_1]$ is equivalent to $A_2 \ge 0$ and $a_0 \in [a^{(0)}_2, a^{(0)}_1]$: if either of these holds, the other one holds too and the cubic (and, of course, its reciprocal) will have three real roots.

If both quadratic elements are negative, the cubic will have one real root, as already shown.

However, negativity of either of the two quadratic elements of the cubic does not mean that the other quadratic element is also negative. Clearly, the negative quadratic element leads straight away to a situation with one real root and a pair of complex conjugate roots. Hence, the reciprocal cubic should be with the same root structure and, if the associated quadratic element is non-negative, then the adjacent coefficient must necessarily lie outside of its prescribed interval. Otherwise, the reciprocal cubic will have three real roots and this is, obviously, impossible.

Hence, the possible situations for the quadratic elements of the cubic and their adjacent coefficients can be summarised as follows (see also Figure 2a and Figure 2b).
\begin{itemize}
\item[(i)] Both quadratic elements positive and both adjacent coefficients lying in their respective prescribed intervals, resulting in the cubic having three real roots;

\item[(ii)] Both quadratic elements negative (the respective intervals for the adjacent coefficients are not over the reals in such case) and the cubic having one real root and a pair of complex conjugate roots;

\item[(iii)] One quadratic element negative and the other quadratic element positive with the adjacent coefficient of the latter not in its relevant prescribed interval, yielding a cubic also with one real root and a pair of complex conjugate roots;

\item[(iv)] Both quadratic elements positive with both adjacent coefficients not lying in their respective prescribed intervals --- again, a cubic with one real root and a pair of complex-conjugate roots.

The question of what happens in the case of a vanishing quadratic element should also be addressed. If, for the cubic polynomial $c(x)$, it happens so that $A_2 = 0$, then, if also $a_0 = a_2^3/a_3^2$, see (\ref{c12}), the cubic polynomial will have a triple real root equal to $-a_2/a_3$; for any other value of $a_0$, the cubic polynomial will have one real root and a pair of complex conjugate roots, see also \cite{29}. Hence, if the cubic polynomial is an element of the cubic sectors of a polynomial of degree 4 or higher, then the higher-degree polynomial will have a reduced number of stationary points (not counted with their individual multiplicities) and, thus, a vanishing quadratic element of the cubic associates with complex roots up in the hierarchy of polynomials of higher degrees.

\end{itemize}

\begin{center}
\begin{tabular}{cc}
\includegraphics[width=68mm, height=58mm]{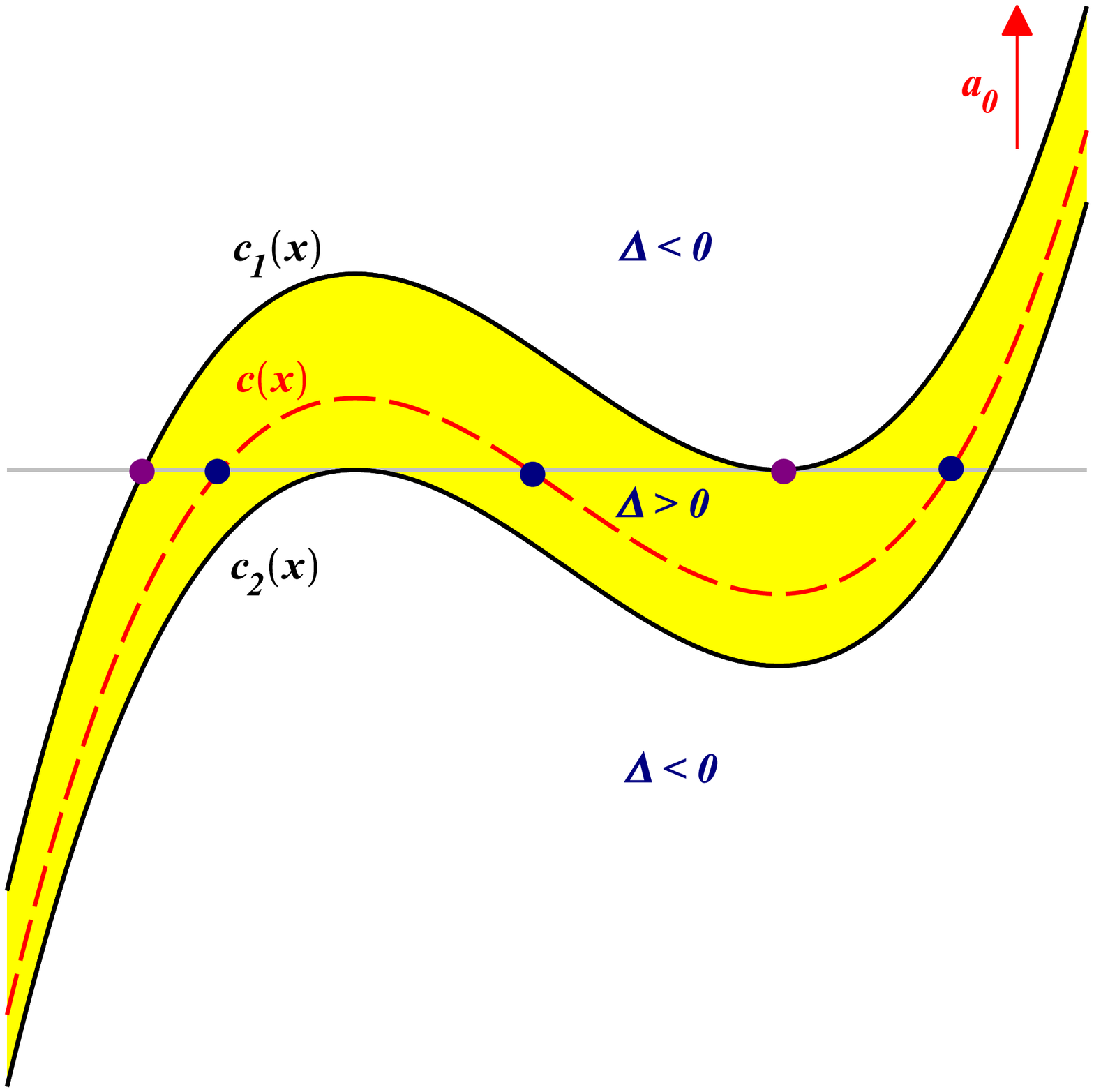} & \includegraphics[width=68mm, height=58mm]{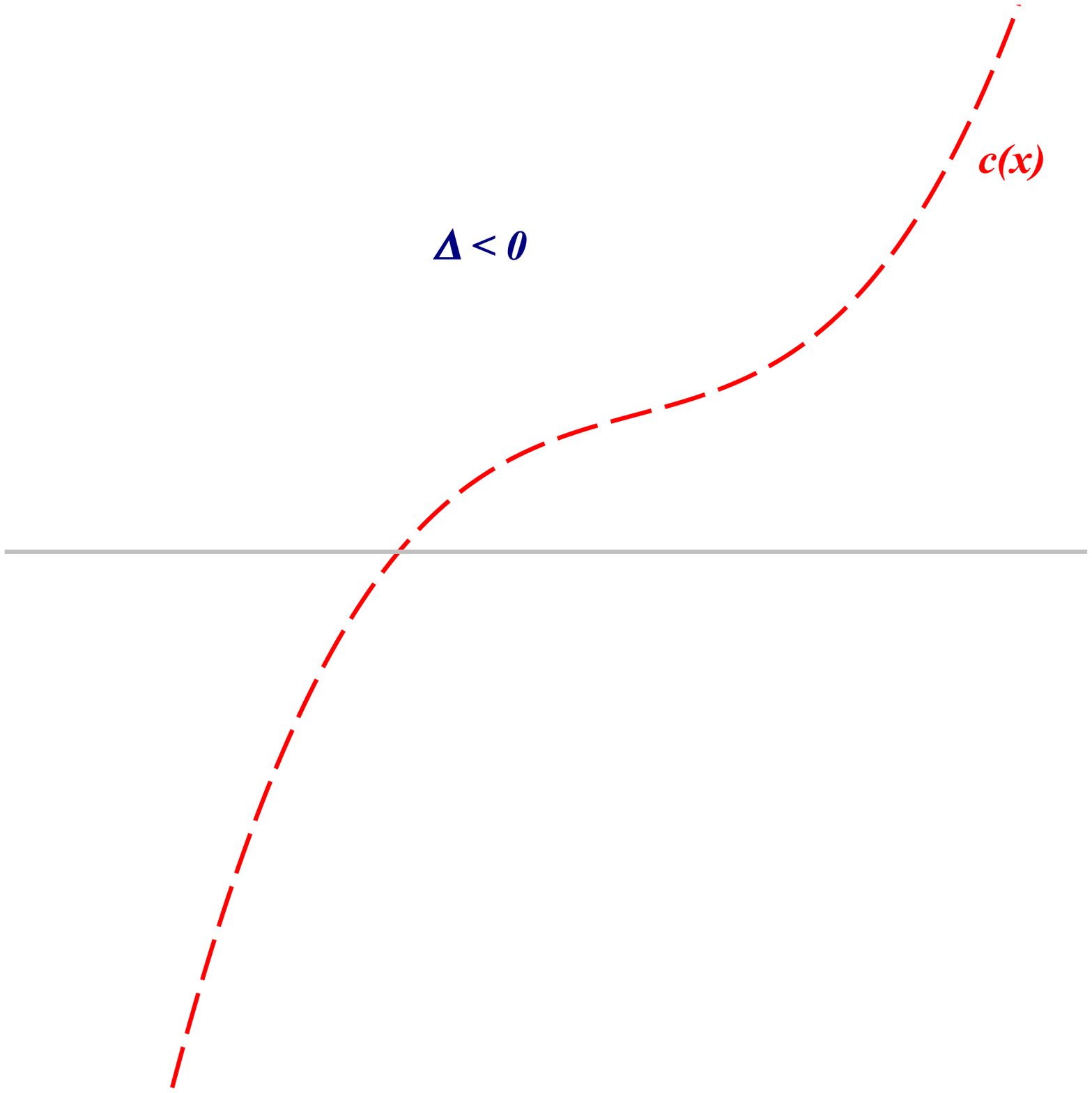}  \\
{\scriptsize {\bf Figure 2a}} &  {\scriptsize {\bf Figure 2b}} \\
\multicolumn{1}{c}
{\begin{minipage}{18em}
\scriptsize
The quadratic element $A_2$ of the cubic $a_3 x^3 + 3 a_2 x^2 + 3 a_1 x + a_0$ is positive and the variation of adjacent coefficient (the free term) $a_0$ results in variation of the sign of the cubic discriminant. When $a_0 = a^{(0)}_{1,2}$, the cubic $c_{1,2}(x)$ is obtained and the cubic discriminant is zero in this case. All cubics in the highlighted band between $c_1(x)$ and $c_2(x)$ (as the dashed curve), that is, all cubics with positive quadratic element $A_2$ and $a_0 \in [a^{(0)}_2, a^{(0)}_1]$, have positive discriminants and three real roots. All other cubics (outside the highlighted band) have negative discriminant and one real root together with a pair of complex conjugate roots. \\
Plotted here as illustration is the polynomial $x^3 - 5x^2 - x + 10$ (dashed curve) and the two {\it privileged} polynomials $c_{1,2}(x)$ with free terms $21.901$ and $-0.049$, respectively. Cubics $x^3 - 5x^2 - x + a_0$ with $a_0 > 21.901$ are in the polynomial band ``above" $c_1(x)$ and have negative discriminant. All of them have one real root and a pair of complex-conjugate roots. In the polynomial band between $c_1(x)$ and $c_2(x)$, that is, for $-0.049 < a_0 < 21.901$, all polynomials have positive discriminant and 3 real roots. Finally, in the polynomial band ``below" $c_2(x)$, namely, for $a_0 < -0.049$, the polynomials again have negative discriminant and one real root and a pair of complex-conjugate roots.
\end{minipage}}
&
\multicolumn{1}{c}
{\begin{minipage}{18em}
\scriptsize
\vskip-3.20cm
This is the cubic $a_3 x^3 + 3 a_2 x^2 + 3 a_1 x + a_0$ with negative quadratic element $A_2$. The discriminant of the cubic is also negative. The derivative of the cubic, that is, the quadratic polynomial $a_3 x^2 + 2 a_2 x + a_1$, has negative discriminant too and thus, it has a pair of complex conjugate roots. Namely, the cubic has no stationary points and, hence, there are no {\it privileged} polynomials associated with it. In result, now there is only one polynomial band --- all polynomials have negative discriminant --- whatever value $a_0$ takes. All these cubics have one real root and a pair of complex conjugate roots. \\
Plotted as illustration is the polynomial $x^3 - x^2 + x + 1$. Cubics $x^3 - x^2 + x + a_0$, with any $a_0$, have negative discriminants. All these cubics have one real root and a pair of complex-conjugate roots.
\end{minipage}} \\
& \\
\end{tabular}
\end{center}

Returning to the positive quadratic elements, it has already been shown that a positive quadratic element, considered in its own right within a cubic, does not necessarily associate with real roots: one has to also consider its adjacent coefficient. For a polynomial $p(x)$ of degree $n \ge 4$, there are $n-2$ ``overlapping" cubic sectors, each associated with two quadratic elements (a cubic polynomial coincides with its only cubic sector). All quadratic elements, except the ones at either ``end" of the polynomial $p(x)$, namely, $A_1 = a_1^2 - a_0 a_2$ and $A_{n-1} = a_{n-1}^2 - a_n a_{n-2}$, are shared between neighbouring cubic sectors. For example, $A_2 = a_2^2 - a_1 a_3$ is, on the one hand, one of the two quadratic elements of the cubic sector $a_3 x^3 + 3 a_2 x^2 + 3 a_1 x + a_0$ (the other one being $A_1 = a_1^2 - a_0 a_2$), with adjacent coefficient $a_0$, while, on the other hand, $A_2$ is also one of the two quadratic elements of the neighbouring cubic sector $a_4 x^3 + 3 a_3 x^2 + 3 a_2 x + a_1$ (the other one being $A_3 = a_3^2 - a_2 a_4$), with adjacent coefficient $a_4$. That is, $A_2$ is ``shared" between these two neighbouring cubic sectors. Note that each of the shared quadratic elements associates with two different adjacent coefficients --- one on either ``side" of it, that is, one from each of the two neighbouring cubic sectors which share the quadratic element. This is exactly what will be considered next: the ``resolution" of a polynomial of degree $n$ into its cubic sectors and investigating whether the roots of these cubic sectors are all real or not.

\subsection{The Cubic Sectors of a Polynomial}

The cubic sectors of a polynomial are obtained by taking four consecutive coefficients and using them to write cubic polynomials with coefficients also in binomial form. This is achieved by a suitable sequence of differentiations and passages to reciprocal polynomials. Differentiating $p(x)$ $n-3$ times yields the cubic sector $c_{n-3}(x) = (n!/3!) (a_n x^3 + 3 a_{n-1} x^2 + 3 a_{n-2} x + a_{n-3})$. Differentiating the reciprocal polynomial $p^\dagger(x) = {n \choose n} a_0 x^n + {n \choose n-1} a_{1} x^{n-1} + \ldots + {n \choose 0} a_n$ gives ${n \choose n} n a_{0} x^{n-1} + {n \choose n-1}(n-1) a_{1} x^{n-2} + {n \choose n-2} (n-2) a_{2} x^{n-3} + \ldots + {n \choose 1} a_{n-1} = n [{n-1 \choose n-1} a_{0} x^{n-1} + {n-1 \choose n-2} a_{1} x^{n-2} + {n-1 \choose n-3} a_{2} x^{n-3} + \ldots + {n-1  \choose 0} a_{n-1}]$. Taking the reciprocal polynomial of this derivative, gives the polynomial of degree $n-1$, $\,\,  n [{n-1  \choose 0} a_{n-1} x^{n-1} + {n-1 \choose 1} a_{n-2} x^{n-2} + \ldots +{n-1 \choose n-1} a_0]$, and differentiating it $n-4$ times yields the cubic sector $c_{n-4}(x) = (n!/3!) (a_{n-1} x^3 + 3 a_{n-2} x^2 + 3 a_{n-3} x + a_{n-4})$. Descending down the ladder to a polynomial of degree $n-2$, gives the next cubic sector: $c_{n-5}(x) = (n!/3!)(a_{n-2} x^3 + 3 a_{n-3} x^2 + 3 a_{n-4} x + a_{n-5})$. The two ``bottom rungs" are clearly $(n!/4!) [a_4 x^4 + 4 a_3 x^3 + 6 a_2 x^2 + 4 a_1 x + a_0]$, from which, after differentiation, one gets the second last cubic sector $c_{1}(x) = (n!/3!)(a_{4} x^3 + 3 a_{3} x^2 + 3 a_{2} x + a_{1})$, and $(n!/3!) [a_3 x^3 + 3 a_2 x^2 + 3 a_1 x + a_0]$ which, itself, is the last cubic sector $c_0(x)$.

As already discussed, if one of the two quadratic elements of a cubic sector is positive and the adjacent coefficient from this cubic sector lies in its prescribed interval, the cubic sector will have three real roots. Hence, the other quadratic element of the same cubic sector must also be positive and its adjacent coefficient from this cubic sector must also lie in its relevant prescribed interval. On the other hand, if one of the two quadratic elements of a cubic sector is positive, but the adjacent coefficient from this cubic sector is not in its relevant prescribed interval, the cubic sector will have only one real root and, hence, the other quadratic element of this cubic sector will be either negative or will be positive and its adjacent coefficient from this cubic sector will also be outside its prescribed interval. Due to the overlap of the cubic sectors, negativity of a quadratic element in one cubic sector ``spills over" into the neighbouring cubic sector by the shared quadratic element in the following manner: a negative quadratic element in a cubic sector leads to either a negative quadratic element in the cubic sector adjacent to the terms of this quadratic element, or to a positive one for which the associated adjacent coefficient does not lie in its prescribed relevant interval. If the ``shared" quadratic element is also negative, then the neighbouring cubic sector will also have one real root only. However, if the ``shared" quadratic element is positive, then, as it associates with a different adjacent coefficient in the neighbouring cubic sector, this different adjacent coefficient in the neighbouring cubic sector may lie or may not lie in its relevant prescribed interval.

Hence, if there is a negative quadratic element, then its neighbouring quadratic element will either be also negative or will be non-negative and its adjacent coefficient will not lie in its relevant prescribed interval. In other words, a positive quadratic element with its adjacent coefficient in its relevant interval can only be a neighbour to a positive quadratic element and can not be a neighbour to a negative quadratic element.

\begin{definition}
A positive quadratic element, for which one or both of its adjacent coefficients do not lie in the relevant prescribed intervals is called a \underline{falsely positive quadratic} \underline{element}.
\end{definition}
A falsely positive quadratic element is clearly associated with complex roots. The change of the sign of some of these, under certain conditions, results in a modification of Newton's Rules.

\section{Modification of Newton’s Rules}

The arguments put forward so far allow one to propose a modification of Newton's Incomplete and Complete Rules. This involves changing the signs of some of the {\it falsely positive quadratic elements}. This should be done under strict conditions and as follows.

One or more falsely positive quadratic element can be neighboured by negative quadratic elements on both sides. In this case, the sign of the falsely positive quadratic element(s) should not be changed as, by doing so, the neighbouring negative groups would be merged and, in result, more permanencies in the sequence of quadratic elements would ensue and thus, the minimum number of complex roots would decrease, rather than increase, while the number of real roots would increase. The Newton rules in this case are not modified.

When one or more falsely positive quadratic elements are neighboured by a positive quadratic element on one side and a negative quadratic element on the other, changing the sign of the falsely positive quadratic element(s) would form a negative group with the neighbouring negative quadratic element or enlarge the existing neighbouring negative group. There would be no change in the minimum number of complex roots, under Newton's incomplete rule, as advocated by Sylvester \cite{syl}. Hence, the sign of the falsely positive quadratic element(s) should also not be changed. Newton's rules will not be modified in this case either.

The proposed minor modification of Newton’s rule is to be applied only when one or more falsely positive quadratic elements are neighboured by positive quadratic elements or by groups of positive quadratic elements. Only in such cases, the signs of these falsely positive quadratic elements should be changed. In result, an existing positive group (of at least three) will be broken by the introduction of a negative quadratic element or a group of negative quadratic elements within it. Hence, the number of permanencies in the sequence of quadratic elements will decrease by two and, with it, the maximum number of real roots will be reduced by two and the minimum number of complex roots will increase by two. This is just another way of saying that falsely positive quadratic elements associate with complex roots, not with real ones. This proposed modification of the Newton Rules can be summarized in the following

\begin{theorem}[Modification of Newton's Rules]
Let $A = \{ A_0 = a_0^2 > 0, \,  A_1 = a_1^2 - a_0 a_2, \, \ldots A_{n-1} = a_{n-1}^2 - a_{n-2}a_n, \, A_n = a_n^2 > 0 \}$ be the sequence of the quadratic elements of the polynomial $p(x)$ as given in binomial from by (\ref{p}).

Let $\widetilde{A}$ be the sequence of the quadratic elements of the polynomial for which the signs in each uninterrupted sub-sequence of one or more falsely positive quadratic elements, which are bounded by positive quadratic elements, have all been changed from positive to negative.

Then the upper bound on the number of real roots of the polynomial is equal to the one obtained by the original Newton's rule on the sequence $A$, reduced by twice the number of such uninterrupted sub-sequences of one or more falsely positive quadratic elements, which are bounded by positive quadratic elements, that is, the upper bound on the number of real roots is obtained by applying the original Newton's rule on the sequence $\widetilde{A}$.
\end{theorem}

\begin{proof}
Due to Corollary 2.1 (Sylvester, \cite{syl}), the number $\rho$ of real roots is less than or equal to the number $r$ of permanencies in the sequence $A$ of quadratic elements, i.e. $\rho \le r$. Let $\tilde{r}$ be the number of permanencies in the sequence $\widetilde{A}$. Suppose there is only one uninterrupted sub-sequence of one or more falsely positive quadratic elements, which are bounded by positive quadratic elements. Under the original Newton's rule, all the signs in this group are positive. Under the proposed modification of the Newton's Rule, the signs of all falsely positive quadratic elements are changed to negative ones. This introduces a ``negative group" within the uninterrupted sub-sequence of one or more falsely positive quadratic elements. Hence, the number of permanencies in the entire sequence of quadratic elements will be reduced by two: the sequence $\widetilde{A}$ has two less permanencies than the sequence $A$. Therefore, there are at least two fewer real roots than the number which can be determined from the original Newton's method, that is $\rho \le \tilde{r} = r - 2$. Applying this argument to all such uninterrupted sub-sequences of one or more falsely positive quadratic elements and supposing that there are $\sigma$ such groups, the upper bound of the total number of real roots will be $\tilde{r} = r - 2 \sigma$. This ``additivity" property of the different newly-formed negative groups follows from Sylvester's statement \cite{syl}, quoted earlier in Section 3 after Theorem 2: ``{\it The number of imaginary roots of an algebraic function cannot be less than the number of negative groups in the complete series of its quadratic elements”}.
\end{proof}

With the modification of Newton's method, one cannot determine by what amount the possible number of positive roots and the possible number of negative roots will be reduced --- it all depends on the sequence of simple elements, which is unaltered by the proposed modification.

\section{Example and Comparison to Other Methods}

The polynomial $x^8 - 16 x^7 + 28 x^6 +112 x^5 - 70 x^4 + (28/5) x^3 + 28 x^2 + 16 x + 1$ has two positive roots ($4.310$ and $13.290$), two negative roots ($-0.071$ and $-2.147$) and two pairs of complex conjugate roots ($0.620 \pm 0.535 i$ and $-0.310 \pm 0.274 i$). The coefficients of the polynomial in binomial form are $a_8 = 1, \, a_7 = -2, \, a_6 = 1, \, a_5 = 2, \, a_4 = -1, \, a_3 = 1/10, \, a_2 = 1, \, a_1 = 2, $ and $a_0 = 1$. The six cubic sectors of the polynomial are: $c_1(x) = (1/10)x^3 + 3 x^2 + 6 x + 1$ having three real roots,  $c_2(x) = - x^3 + (3/10) x^2 + 3 x + 2$ with only one real root, $c_3(x) = 2 x^3 - 3 x^2 + (3/10) x + 1$ with only one real root, $c_4(x) = x^3 + 6 x^2 - 3 x + 1/10$ with three real roots, $c_5(x) = -2 x^3 + 3 x^2 + 6 x - 1$ with three real roots, and $c_6(x) = x^3 - 6 x^2 + 3 x + 2$ also with three real roots. All quadratic elements of this polynomial are positive: $A_8 = 1, \, A_7 = 3, \, A_6 = 5, \, A_5 = 5, \, A_4 = 4/5, \, A_3 = 101/100, \, A_2 = 4/5, \, A_1 = 3, $ and $A_0 = 1$. However, none of the two coefficients $a_1 = 2$ [from $c_2(x)$] and $a_5 = 2$ [from $c_3(x)$], adjacent to $A_3$ [which is shared by $c_2(x)$ and $c_3(x)$ --- both with a single real root], lies in its relevant prescribed interval. Also, for the quadratic element $A_2$ [shared by $c_1(x)$ and $c_2(x)$ --- the first with three real roots, the second with just one], the adjacent coefficients $a_0 = 1$ [from $c_1(x)$] and $a_4 = -1$ [from $c_2(x)$] are such that $a_0$ belongs to its prescribed interval, while $a_4$ does not. Likewise, for the quadratic element $A_4$ [shared by $c_3(x)$ and $c_4(x)$ --- the first with one real root and the second with three], the two adjacent coefficients $a_2 = 1$ [from $c_3(x)$] and $a_6 = 1$ [from $c_4(x)$] are such that $a_2$ does not belong to its prescribed interval, while $a_6$ does. Hence, $A_2, \,\, A_3$ and $A_4$ are all {\it falsely positive quadratic elements}. As these form a group and this group is neighboured by positive quadratic elements ($A_1 = 3$ and $A_5 = 5$), the signs of $A_2, \,\, A_3$ and $A_4$ should be all changed from plus to minus.

If the original Newton's Rule is applied, as all quadratic elements are positive, the number of sign changes in the sequence of coefficients (simple elements) gives an upper bound on the number of positive roots. There are four sign changes, hence, there are maximum 4 positive roots. The number of permanencies in the sequence of coefficients gives an upper bound on the number of negative roots. There are 4 permanencies, hence there are maximum 4 negative roots.

With the proposed modification, the signs of $A_2, \,\, A_3$ and $A_4$ are now negative. The number of permanencies in the sequence of quadratic elements is now 6, not 8. Applying the Complete Newton's Rule to this modified sequence of quadratic elements yields maximum 3 positive roots and maximum 3 negative roots. These are  stricter bounds on the positive and negative roots than the ones provided by the unmodified sequence of quadratic elements, as in Newton's original Rule.

One can also compare this result with what would be obtained if the Descartes' rule of signs is used. The sequence of the coefficients of the polynomial exhibits 4 sign changes, thus the number of positive roots can be either 4, or 2, or 0. Replacing $x$ with $-x$ yields the polynomial $p(-x) = x^8 + 16 x^7 + 28 x^6 - 112 x^5 - 70 x^4 - (28/5)x^3 + 28 x^2 - 16 x + 1$. There are four sign changes in the sequence of the coefficients of the polynomial $p(-x)$, hence the number of negative roots is either 4, or 2, or 0.

If the Descartes' rule of signs is used in conjunction with the modified Newton's rule, one obtains that the number of positive roots in this example are either 2 or 0 and that the number of negative roots is also either 2 or 0.

Fourier's theorem \cite{f} (also referred to as Budan--Fourier Theorem or Budan Theorem) is \cite{akr}:

\begin{theorem}[Fourier]
If in the sequence of $n+1$ functions given by $S(x) = \{ p(x), \, p'(x), $ \linebreak $ \, p''(x), \ldots, p^{(n)}(x) \}$, where $p^{(n)}(x)$ is the $n^{\mbox{\scriptsize th}}$ derivative of $p(x)$, one first replaces $x$ by $q$ and then replaces $x$ by $r$, where $q < r$, then the following holds for the two sequences $S(q)$ and $S(r)$:
\begin{itemize}
\item[(i)] The sequence $S(r)$ cannot have more sign variations than the sequence $S(q)$.
\item[(ii)] The number $\rho$ of real roots of the equation $p(x) = 0$, located between $q$ and $r$, cannot be greater than the number $v$ of sign variations lost in passing from the sequence $S(q)$ to the sequence $S(r)$.
\item[(iii)] When $\rho < v$, the difference is an even number.
\end{itemize}
\end{theorem}

The signs of the Fourier sequence at $x = -3$ are $\{+, -, +, -, +, -, +, -, +\}$. At $x = 0$, they are $\{+, +, +, +, -, +, +, -, +\}$. There are 8 sign variations in the first sequence and 4 sign variations on the second one. The difference is 4. Hence, the number of real roots between $-3$ and $0$ is either 4, or 2, or zero. The signs of the Fourier sequence at $x = 15$ are all positive. The difference between the sign variations in the sequence $S(0)$ and those in the sequence $S(15)$ is again 4.  Hence, the number of real roots between $0$ and $15$ is either 4, or 2, or zero.

One can also use Sturm's theorem \cite{sturm}.
\begin{theorem}[Sturm]
Consider the sequence $\widetilde{S}(x) = \{ p_0(x), p_1(x), \ldots , p_n(x) \}$, in which $p_0(x) = p(x),  p_1(x) = p'(x), p_2(x) =  - \mbox{rem} \{ [p_0(x)]/[p_1(x)] \}, \ldots, p_n(x) = - \mbox{rem} \{ [p_{n-2}(x)]$ $/[p_{n-1}(x)] \},$ where rem$\{[p_i(x)]/[p_{i+1}(x)]\}$ denotes the remainder of the division of the polynomial $p_i(x)$ by the polynomial $p_{i+1}(x)$.

If $p(x)$ has only simple zeros, then the number of real roots on the interval $(q, r)$ is equal to the difference of sign variations in the sequences $\widetilde{S}(q)$ and $\widetilde{S}(r)$.
\end{theorem}

The signs of the Sturm sequence at $x = - 3$ are $\{ +, -, +, -, +, -, -, -, + \}$ and those at $x = 0$ are $\{ +, +, -, +, +, -, -, +, + \}$. There are 6 sign changes in  $\widetilde{S}(-3)$ and four sign changes in  $\widetilde{S}(0)$. The difference is 2. Hence there are 2 real roots between $-3$ and $0$.

The signs of the Sturm sequence at $x = 15$ are $\{ +, +, +, +, +, +, -, +, + \}$. There are 2 sign changes in it. As the number of sign changes in $\widetilde{S}(0)$ is 4 and this exceeds the number of sign changes in $\widetilde{S}(15)$ by 2, then there are 2 real roots between $0$ and $15$.

Of course, one can use $-\infty$ instead of $q$ and $+\infty$ instead of $r$. Then only the leading terms in each element of the Sturm sequence is relevant, with $\widetilde{S}(-\infty) - \widetilde{S}(0)$ giving the number of negative roots (which is 2) and $\widetilde{S}(0) - \widetilde{S}(\infty)$ giving the number of positive roots (also 2).

Despite the fact that Sturm's theorem gives the exact number of roots, this method, in comparison to the modified Newton's method, is very intensive labour-wise.

\section{Example with Parametric Polynomial --- Rayleigh Waves}

The Rayleigh waves are elastic surface waves which propagate on the surface of solids \cite{landau}. These are superposition of two waves which are propagated independently and the equation of motion is \cite{landau}:
\b
\label{eom}
\frac{\partial^2 \vec{u}}{\partial t^2} - c^2 \Delta \vec{u} = 0,
\e
where $\vec{u} = \vec{u}_l + \vec{u}_t$ is the displacement vector with $\vec{u}_l$ being the displacement in the direction of propagation (longitudinal wave) and $\vec{u}_t$ being the displacement in a plane, perpendicular to the direction of propagation (transverse wave).

The transverse wave $\vec{u}_t$ satisfies div$\,\vec{u}_t = 0$, while for the longitudinal wave $\vec{u}_l$ one has curl$\,\vec{u}_l = 0$. These yield \cite{landau}:
\b
\label{eqqs1}
a (k^2 - \kappa_t^2) + 2 b k \kappa_l & = & 0, \\
2 a k \kappa_t + b (k^2 + \kappa_t^2) & = & 0,
\e
where $a$ and $b$ are some constants, $k$ is the wave number, $\kappa_t = \sqrt{k^2 - \omega^2/c_t^2}$ and $\kappa_l = \sqrt{k^2 - \omega^2/c_l^2}$ are the rapidities of the transverse and longitudinal damping. Here $\omega$ is the angular frequency and $c_l$ and $c_t$ are the longitudinal and transverse speeds of sound \cite{landau}. The condition for compatibility of equations in (\ref{eqqs1}) and (29) is given by \cite{landau}:
\b
\label{cond}
\left( 2 k^2 - \frac{\omega^2}{c_t^2} \right)^4 = 16 k^4 \left( k^2 - \frac{\omega^2}{c_t^2} \right) \left( k^2 - \frac{\omega^2}{c_l^2} \right).
\e
Introducing $\xi = \omega/(c_t k) > 0$, $x = \xi^2$, and $q = c_t^2/c_l^2$, from (\ref{cond}), one gets the cubic equation \cite{landau}
\b
\label{land}
x^3 + 3 \!\left(-\frac{8}{3} \right) \! x^2 + 3 \, \frac{8 (3 - 2q)}{3} \, x - 16(1 -q) = 0
\e
--- written, conveniently, in binomial form.
The allowed values of the parameter $q$ satisfy $0 \le q < 3/4$ --- see \cite{landau}. 

The conditions on the parameter $q$ for one or three real roots of this cubic equation and the isolation intervals of the roots have been determined in \cite{27}. 

The sequence of simple elements for the cubic (\ref{land}) is:
\b
a_3 = 1 > 0, \quad a_2 = -\frac{8}{3} < 0, \quad a_1 = \frac{8}{3} (3 - 2q) > 0, \quad a_0 = - 16(1 - q) < 0
\e
--- given that $0 \le q < 3/4$, \, $a_1$ is positive for $q < 3/2$, while $a_0$ is negative for $q < 1$.
The sequence of quadratic elements for the cubic (\ref{land}) is:
\b
A_3 & \!\! = \!\! & a_3^2 \, = \, 1 > 0, \qquad A_2 \, = \, a_2^2 - a_1 a_3 \, = \, \frac{8}{9} (6q -1), \nonumber \\
A_1 & \!\! = \!\! & a_1^2 - a_0 a_2 \, = \, \frac{64}{9} (4q^2 - 6q +3) > 0, \qquad A_0 \, = \, a_0^2 \, = \, 256(1 - q)^2 > 0.
\e
Clearly, $A_2 > 0$ when $q > 1/6$.

Applying the original Newton's rule on the double sequence \raisebox{-.1cm}{$\stackrel{\scriptstyle{a_3, \,\, a_2, \,\, a_1, \,\, a_0}}{\scriptstyle{A_3, \, A_2, \, A_1, \, A_0}}$}, \,  one immediately determines that when $q > 1/6$, all quadratic elements are positive. The signs in this double sequence are \raisebox{-.1cm}{$\stackrel{\scriptstyle{+, -, +, -}}{\scriptstyle{+, +, +, +}}$}, yielding 3 permanencies in the sequence of quadratic elements and 3 variations in the sequence of simple elements. Hence, the number of real roots is 3 or 1. They are all positive, as there are no double permanencies. On the other hand, when $q < 1/6$, one has $A_2 < 0$ and the signs in the double sequence in this case are \raisebox{-.1cm}{$\stackrel{\scriptstyle{+, -, +, -}}{\scriptstyle{+, +, -, +}}$}. There is only one variation permanence and this indicates one positive root. There are no double permanencies, hence there are no negative roots.

Consider next whether the modified Newton's rule applies for this example, that is, whether falsely positive quadratic elements exist. Indeed, this is the case: the quadratic elements $A_2 = a_2^2 - a_1 a_3 = (8/9) (6q -1)$ and $A_1 =  a_2^2 - a_1 a_3 = (64/9) (4q^2 - 6q +3)$ are neighboured by two positive quadratic elements ($A_3$ and $A_0$) and are themselves both positive for $q > 1/6$ (the latter is positive for all $q$). However, for certain values of $q > 1/6$, it will turn out that the quadratic elements $A_2$ and $A_1$ are both, {\it simultaneously}, falsely positive and their signs should be changed to minus.

The adjacent coefficient $a_0 = - 16(1 - q)$ of the quadratic element $A_2$ has its prescribed interval given in {\bf Definition 8}, namely:
\b
\left[ \frac{- 3 a_2 A_2 + a_2^3 - 2 A_2^{3/2}}{a_3^2},  \frac{- 3 a_2 A_2 + a_2^3 + 2 A_2^{3/2}}{a_3^2} \right]   \hskip4.5cm \phantom{emp} \nonumber \\
= \left[ \frac{64}{27}(18q - 11) - \frac{32\sqrt{2}}{27} (6q - 1) \sqrt{6q - 1}, \,\,\, \frac{64}{27}(18q - 11) + \frac{32\sqrt{2}}{27} (6q - 1) \sqrt{6q - 1} \right]
\e
--- see also (\ref{c12}). 

It is easy to see that $a_0 > (64/27)(18q - 11) - (32\sqrt{2}/27) (6q - 1) \sqrt{6q - 1}$ for all $q$, while $a_0 < (64/27)(18q - 11) + (32\sqrt{2}/27) (6q - 1) \sqrt{6q - 1}$ for $q < \widetilde{q}$, where $\widetilde{q}$ is the single real root of
\b
\label{qqq}
64 q^3 - 107 q^2 + 62 q - 11 = 0,
\e
namely $\widetilde{q} \approx 0.3215 > 1/6$ (the other two roots are $0.6752 \pm 0.2806i$).

Then, for $1/6 < q < \widetilde{q}$, the adjacent coefficient $a_0 = - 16(1 - q)$ of the quadratic element $A_2$ is not in its prescribed interval, hence the quadratic element $A_2$ is falsely positive and its sign should be changed to minus. 

If $q$ is such that the quadratic element $A_2$ is falsely positive, resulting in a cubic with one real root only, then the reciprocal cubic will also have one real root only. Hence, over the same range $1/6 < q < \widetilde{q}$, the quadratic element $A_1 =  a_2^2 - a_1 a_3 = (64/9) (4q^2 - 6q +3)$ (which is positive for all $q$) will also be falsely positive --- its adjacent coefficient $a_0$ will not belong to its prescribed interval. Thus, the sign of $A_1$ must also be changed to minus over the interval $1/6 < q < \widetilde{q}$.  Therefore, in the case of $1/6 < q < \widetilde{q}$, one should instead use the following sequence of quadratic elements: $\widetilde{A}_3 = A_3, \, \widetilde{A}_2 = - A_2, \, \widetilde{A}_1 = - A_1, \, \widetilde{A}_0 = A_0$, that is, the modified Newton's rule applies. Namely, one should instead use the double sequence \raisebox{-.1cm}{$\stackrel{\scriptstyle{a_3, \,\, a_2, \,\, a_1, \,\, a_0}}{\scriptstyle{\widetilde{A}_3, \, \widetilde{A}_2, \, \widetilde{A}_1, \, \widetilde{A}_0}}$}, the signs in which are \raisebox{-.1cm}{$\stackrel{\scriptstyle{+, -, +, -}}{\scriptstyle{+, -, -, +}}$}. There is only one variation permanence, indicating one positive root, and no double permanencies, hence no negative roots --- exactly as in the case of $q < 1/6$. For example, taking $q = 0.3210$ --- just under $\widetilde{q} = 0.3215$ --- yields the following roots of (\ref{land}): $0.8494, \,\, 3.5753 \pm 0.0869i$.

On the other hand, for $q > \widetilde{q}$, the adjacent coefficient $a_0 = - 16(1 - q)$ of the quadratic element $A_2$ belongs to its prescribed interval, hence the quadratic element $A_2$ is truly positive and its sign should not be changed. Similarly, $A_1$ will also be truly positive --- otherwise the equation and its reciprocal will not have the same number of real roots, which is impossible. The situation in this case is exactly the same as that in the original Newton's rule for $q > 1/6$: all quadratic elements are positive, i.e. there are three permanencies in their sequence, while there are three variations in the sequence of simple elements. Hence, the number of real roots is 3 or 1. Indeed, taking $q = 0.3220$ --- just over $\widetilde{q} = 0.3215$ --- yields the following roots of (\ref{land}): $0.8491, \,\, 3.4884, \,\, 3.6625$.

When  $q = \widetilde{q}$, equation (\ref{land}) has 3 real roots, two of which equal: $0.8492,  3.5754, \, 3.5754$.

It should be noted that, given the presence of the parameter $q$, the use of the Cardano formul\ae\, would not yield helpful information for the roots of equation (\ref{land}). However the cubic (\ref{qqq}) is easily solvable and in result, the modified Newton's rule gives stricter limit on the number of real roots.

\section{Algorithm for Applying the Modified Newton's Method}

\begin{enumerate}
\item [(i)] Input the coefficients $\alpha_j, \,\,\, j = 0, 1, \ldots , n$ (with $\alpha_n \ne 0$), of the polynomial $p_n(x) = \alpha_0 + \alpha_1 x + \cdots + \alpha_n x^n$.
\item [(ii)] Pass into the equivalent polynomial in binomial form $p_n(x) = \sum_{k=0}^n {n \choose k} a_k x^k$, where $a_k = \alpha_k / {n \choose k}$.
\item [(iii)] Calculate the quadratic elements $A_0 = a_0^2, \,\, A_j = a_j^2 - a_{j-1} a_{j+1}$ for $j = 1, 2, \ldots, n-1$, and $A_n = a_n^2$ (noting that $A_0$ and $A_n$ are always positive).
\item [(iv)] Introduce $\vec{a} = (a_0, \, a_1, \, \ldots \, a_n)$ --- the sequence of the coefficients of the polynomial in binomial form (the so-called simple elements). This is an $(n+1)$-vector.
\item [(v)] Introduce $\vec{A} = (A_0, \, A_1, \, \ldots \, A_n)$ --- the sequence of the quadratic elements of the polynomial in binomial form. This is another  $(n+1)$-vector.
\item [(vi)] Determine if in the sequence $\vec{A}$ of quadratic elements there is a group of three or more ($m + 1$) neighbouring positive quadratic elements, that is, find out whether sign$(A_j) = $ sign$(A_{j+1}) = $ sign$(A_{j+2}) = \ldots = $ sign$(A_{j+m}) = 1$ and whether \linebreak sign$(A_{j-1}) =  -1,$ for some $j$ from $0$ to $n-2$ and $j + m \le n$ (with $A_{-1}$ defined as negative). Recall that $A_0$ and $A_n$ are both always positive. If no such sequence exists, Newton's method will not be modified and one should proceed with step (xii).
\item [(vii)] Determine whether the first of these, $A_j$, is not a falsely positive quadratic element. This is done by verifying that both of its adjacent coefficients $a_{j-2}$ and $a_{j+2}$ lie in their prescribed intervals --- respectively, $I_{j-2} = [ (- 3 a_j A_j + a_j^3 - 2 A_j^{3/2} ) /  a_{j+1}^2,   \,\,  (- 3 a_j A_j + a_j^3 + 2 A_j^{3/2} ) / a_{j+1}^2 ]$, for $j = 2, 3, ..., n-1,$ and $K_{j+2} = [ (- 3 a_j A_j + a_j^3 - 2 A_j^{3/2} ) /  a_{j-1}^2, \,\,  (- 3 a_j A_j + a_j^3 + 2 A_j^{3/2} ) / a_{j-1}^2 ]$, for $j = 1, 3, ..., n-2$. Recall that $A_1$ has only one adjacent coefficient ($a_3$) and $A_{n-1}$ also has only one adjacent coefficient ($a_{n-3}$).
\item [(viii)] If $A_j$ is falsely positive, disregard $A_j$ and repeat the previous step (vii) for the next quadratic element, $A_{j+1}$, from the sub-sequence found in step (vi). Keep repeating until a truly positive quadratic elements is found. If no such quadratic element exists, Newton's method will not be modified and one should proceed with step (xii).
\item [(ix)] Once a truly positive quadratic element is found, verify whether the following quadratic element is truly positive. If so, disregard the first quadratic element and keep repeating this step until a falsely positive quadratic element is found which follows a truly positive quadratic element.
\item [(x)] If the sub-sequence of a truly positive quadratic element and a falsely positive quadratic element is followed by a negative quadratic element, disregard that sequence and keep repeating this step until a sequence is found in which the first and the last quadratic elements are truly positive and those in between --- falsely positive. Note that there is at least one truly positive quadratic element until the end ($A_n$ is such). If no such sequence exists, Newton's method will not be modified and one should proceed with step (xii).
\item [(xi)] Change the signs of all falsely positive quadratic elements in the sequence from step (x).
\item [(xii)] Let $\vec{s}\,\, $ be the $n$-vector describing the variances in sign in the sequence $\vec{a}$ of the simple elements, namely: $\vec{s} = (1/2) ( |$sign$(a_1) - $sign$(a_0)|, |$sign$(a_2) - $sign$(a_1)|, \ldots , $\linebreak $ |$sign$(a_n) -$ sign$(a_{n-1})|).$ Note that $\vec{s}$ contains only ones (denoting variances in sign) and zeros (denoting permanencies in sign).
\item [(xiii)] Similarly, let $\vec{S}$ be the $n$-vector describing the variances in sign in the sequence $\vec{A}$ of the quadratic elements (whether modified or not): $\vec{S} = (1/2) ( |$sign$(A_1) - $sign$(A_0)|, $ \linebreak $ |$sign$(A_2) - $sign$(A_1)|, \ldots , |$sign$(A_n) - $sign$(A_{n-1})|).$ In the same manner, $\vec{S}$ contains only ones (denoting variances in sign) and zeros (denoting permanencies in sign).
\item [(xiv)] Calculate $\vec{q} = \vec{s} - \vec{S}$. The $n$-vector $\vec{q}$ contains minus ones (denoting permanence $p$ in the sequence $\vec{a}$ and variance $V$ in the sequence $\vec{A}$); ones (denoting variance $v$ in the sequence $\vec{a}$ and permanence $P$ in the sequence $\vec{A}$); and zeros (denoting either permanence $p$ in the sequence $\vec{a}$ and permanence $P$ in the sequence $\vec{A}$ or variance $v$ in the sequence $\vec{a}$ and variance $V$ in the sequence $\vec{A}$).
\item [(xv)] The number of permanencies-variances ($pV$) is given by $(1/2)\sum_{j=1}^n(q_j - | q_j | )$. The maximum number of positive roots is given by the number $vP$ of variances-permanencies and is equal to $(1/2)\sum_{j=1}^n(q_j + | q_j | )$.
\item [(xvi)] Calculate $\vec{\widetilde{Q}} = \vec{s} + \vec{S}$. The $n$-vector $\vec{\widetilde{Q}}$ contains zeros (denoting permanence $p$ in the sequence $\vec{a}$ and permanence $P$ in the sequence $\vec{A}$); twos (denoting variance $v$ in the sequence $\vec{a}$ and variance $V$ in the sequence $\vec{A}$); and ones (denoting either permanence $p$ in the sequence $\vec{a}$ and variance $V$ in the sequence $\vec{A}$ or variance $v$ in the sequence $\vec{a}$ and permanence $P$ in the sequence $\vec{A}$).
\item [(xvii)] Calculate $\vec{Q} = \vec{\widetilde{Q}} - \vec{e}$, where the $n$-vector $\vec{e}$ contains only ones.
\item [(xviii)] The number of variances-variances ($vV$) is given by $(1/2)\sum_{j=1}^n(Q_j + | Q_j | )$. The maximum number of negative roots is given by the number $pP$ of permanencies-permanencies and is equal to $(1/2)\sum_{j=1}^n(Q_j - | Q_j | )$.
\end{enumerate}

\section{New Necessary Condition for the Reality of All Roots of a Real  Polynomial. Example}

The analysis of the cubic sectors of a real polynomial also allows the formulation of a new necessary condition for the reality of its roots: {\it if the roots of a real polynomial are all real numbers, then each quadratic element of the polynomial is positive and each of its adjacent coefficients lies in its relevant prescribed interval}.

An equivalent form of this new necessary condition is the following: {\it if all of the roots are real, then the polynomial cannot have negative, vanishing, or falsely positive quadratic elements}.

It should be noted that a polynomial cannot have a single falsely positive quadratic element. This can be seen in the following manner. If there is one cubic sector with a falsely positive quadratic element, then the reciprocal cubic polynomial, that is, a neighbouring cubic sector, will have either another falsely positive quadratic element (meaning that the polynomial has more than one falsely positive quadratic elements) or a negative quadratic element. As the passage from a negative quadratic element to a positive quadratic element necessarily happens through a falsely positive quadratic element (with either one or both of its associated adjacent coefficients outside of their relevant prescribed intervals), there cannot be a polynomial with a single falsely positive quadratic element.

As example for this section, consider the quintic polynomial $x^5 + x^4 - 28 x^3 + 32 x^2 + 96 x - 144 = x^5 + 5 (1/5) x^4 + 10 (-28/10) x^3 + 10 (32/10) x^2 + 5 (96/5) x + (-144)$. The latter is in binomial form. All of its roots are real: $x_1 = 3, \,\, x_2 = x_3 = 2, \,\, x_4 = -2,$ and $x_5 = -6$. Hence, all of the quadratic elements are positive: $A_0 = 20736, \,\, A_1 = 20736/25, \,\, A_2 = 64, \,\, A_3 = 36/5, \,\, A_4 = 71/25,$ and $A_5 = 1$.  There are no falsely positive quadratic elements. All this can be seen as follows. The cubic sectors of this quintic polynomial are three. The first of them is $c_1(x) = - (28/10) x^3 + 3(32/10) x^2 + 3(96/5) x - 144$. The two quadratic elements of this cubic sector are $A_1 = (96/5)^2 - (32/10)(-144) =  20736/25$ and $A_2 = (32/10)^2 - (96/5)(-28/10) =  64$. The adjacent coefficient of $A_1$ is $a_3 = -28/10$ and the adjacent coefficient of $A_2$ is $a_0 = -144$. The prescribed interval of the adjacent coefficient $a_3$ is $[ (- 3 a_1 A_1 + a_1^3 - 2 A_1^{3/2})/a_0^2, \,\, (- 3 a_1 A_1 + a_1^3 + 2 A_1^{3/2})/a_0^2] = [-4.267, 0.341]$. The coefficient $a_3$ is $-28/10$ and it belongs to its prescribed interval. The prescribed  interval of the adjacent coefficient $a_0$ is $[ (- 3 a_2 A_2 + a_2^3 - 2 A_2^{3/2})/a_3^2, \,\, (- 3 a_2 A_2 + a_2^3 + 2 A_2^{3/2})/a_3^2] = [-204.800, 56.424]$. The coefficient $a_0$ is $-144$ and it also belongs to its prescribed interval. One should not even have to verify this: as $A_1 > 0$ and as its adjacent coefficient $a_3$ is in its prescribed interval, then the other quadratic element, $A_2$, is also positive and its adjacent coefficient $a_0$ will also be in its prescribed interval --- obvious from the relationship between the roots of reciprocal polynomials. The roots of the first cubic sector are all real: $5.470, \,\, 2.211,$ and $-4.253$.

Similar arguments apply to the other two quadratic sectors. The second of them is $c_2(x) = (1/5) x^3 + 3(-28/10) x^2 + 3 (32/10) x + 96/5$. The two quadratic elements of
this cubic sector are $A_2 = (32/10)^2 - (-28/10)(96/5) =  64$ and $A_3 = (-28/10)^2 - (1/5)(32/10) =  36/5$. The adjacent coefficient of $A_2$ is $a_4 = -28/10$ and the adjacent coefficient of $A_3$ is $a_1 = 96/5$. The prescribed interval of the adjacent coefficient $a_4$ is $[ (- 3 a_2 A_2 + a_2^3 - 2 A_2^{3/2})/a_1^2, \,\, (- 3 a_2 A_2 + a_2^3 + 2 A_2^{3/2})/a_1^2] = [-4.356, 1.200]$. The coefficient $a_4$ is $1/5$ and it belongs to its prescribed interval. The prescribed  interval of the adjacent coefficient $a_1$ is $[ (- 3 a_3 A_3 + a_3^3 - 2 A_3^{3/2})/a_4^2, \,\, (- 3 a_3 A_3 + a_3^3 + 2 A_3^{3/2})/a_4^2] = [-2.781, 1929.182]$. The coefficient $a_1$ is $96/5$ and it also belongs to its prescribed interval. The cubic sectors $c_1(x)$ and $c_2(x)$ ``share" the quadratic element $A_2$. The positivity of $A_2$ and the fact that its adjacent in $c_2(x)$ coefficient $a_4$ is in its prescribed interval, guarantees positivity of $A_3$ and that the coefficient $a_1$, adjacent in $c_2(x)$ to $A_3$, will also be in its prescribed interval. The three real roots of the cubic sector $c_2(x)$ are: $40.760, \,\, 2.272,$ and $-1.037,$

The third cubic sector is $c_3(x) = x^3 + 3(3/15) x^2 + 3(-28/10)x + 32/10$.  The two quadratic elements of this cubic sector are $A_3 = (-28/10)^2 - (3/15)(32/10) =  36/5$ and $A_4 = (3/15)^2 - (1)(-28/10) =  71/25$. The adjacent coefficient of $A_3$ is $a_5 = 1$ and the adjacent coefficient of $A_4$ is $a_2 = 32/10$. The prescribed interval of the adjacent coefficient $a_5$ is $[ (- 3 a_3 A_3 + a_3^3 - 2 A_3^{3/2})/a_2^2, \,\, (- 3 a_3 A_3 + a_3^3 + 2 A_3^{3/2})/a_2^2] = [-0.011, 7.536]$. The coefficient $a_5$ is $1$ and it belongs to its prescribed interval. The prescribed  interval of the adjacent coefficient $a_2$ is $[ (- 3 a_4 A_4 + a_4^3 - 2 A_4^{3/2})/a_5^2, \,\, (- 3 a_4 A_4 + a_4^3 + 2 A_4^{3/2})/a_5^2] = [-11.268, 7.876]$. The coefficient $a_2$ is $32/10$ and it also belongs to its prescribed interval. The cubic sectors $c_2(x)$ and $c_3(x)$ ``share" the quadratic element $A_3$. The positivity of $A_3$ and the fact that its adjacent in $c_3(x)$ coefficient $a_5$ is in its prescribed interval, guarantees the positivity of $A_4$ and that the coefficient $a_2$, adjacent in $c_3(x)$ to $A_4$, will also be in its prescribed interval. The three real roots of the cubic sector $c_3(x)$ are: $0.4000, \,\, 2.372,$ and $-3.372$.

\section{Relationship among the Discriminants of Real Polynomials, the Discriminants of their Derivatives, and the Quadratic Elements}

As seen earlier (in Section 5), the discriminant $\Delta_3$ of the cubic polynomial $a_3 x^3 + 3 a_2 x^2 + 3 a_1 x + a_0$ is quadratic in the free term $a_0$ and the discriminant $\Delta_2$ of the discriminant $\Delta_3$ is proportional to the third power of quadratic element $A_2: \,\, \Delta_2 = 16 A_2^3$. Hence, the sign of the discriminant of the quadratic, which is the discriminant in $a_0$ of the cubic, is determined entirely by the sign of the quadratic element $A_2$. As discussed, if $A_2 \ge 0$, it is the adjacent coefficient, $a_0$, which controls the number of real roots of the cubic.

The discriminant of the derivative of the cubic, that is, the discriminant of $3 (a_3 x^2 + 2 a_2 x + a_1)$ is equal to $-36  A_2$.

For the quartic polynomial $a_4 x^4 + 4 a_3 x^3 + 6 a_2 x^2 + 4 a_1 x + a_0$, the discriminant is $\Delta_4 = 256 [a_4^3 a_0^3 + (- 12 a_1 a_3 a_4^2 - 18 a_2^2  a_4^2 + 54 a_2 a_3^2 a_4 - 27 a_3^4) a_0^2 + (54 a_1^2 a_2 a_4^2 - 6 a_1^2 a_3^2 a_4 - 180 a_1 a_2^2 a_3 a_4 + 108 a_1 a_2 a_3^3 + 81 a_2^4 a_4 - 54 a_2^3 a_3^2) a_0 - 27 a_1^4 a_4^2 + 108 a_1^3 a_2 a_3 a_4 - 64 a_1^3 a_3^3 - 54 a_1^2 a_2^3 a_4 + 36 a_1^2 a_2^2 a_3^2]$ --- a cubic in the free term $a_0$.

The discriminant of this cubic in $a_0$ is $19683 \, (a_1 a_4^2 - 3 a_2 a_3 a_4 + 2 a_3^3)^2 \, (- a_1^2 a_4^2 + 6 a_1 a_2 a_3 a_4 - 4 a_1 a_3^3 - 4 a_2^3 a_4 + 3 a_2^2 a_3^2)^3$. The sign of this discriminant is determined by the sign of the expression that is cubed. This expression is quadratic in $a_1$ and its discriminant is given by $16 A_3^3$.

The discriminant of the derivative of the quartic is $6912 [- a_1^2  a_4^2 + (6 a_2 a_3 a_4 - 4 a_3^3) a_1 - 4 a_2^3 a_4 + 3 a_3^2 a_2^2]$. This is a quadratic in the free term $a_1$ of the derivative of the quartic and the discriminant of this quadratic is $764411904 A_3^3$.

The discriminant of the second derivative of the quartic is $-576 A_3$.

Such type of relationships exist among the discriminants of polynomials of higher degrees, their derivatives, and their quadratic elements.

It is enticing to consider the {\it quartic sectors} (or higher) of a polynomial, instead of its {\it cubic sectors}.

Return to the quartic and recall that the discriminant in $a_0$ of the discriminant $\Delta_4$, which is a cubic in $a_0$, is given by $19683 \, (a_1 a_4^2 - 3 a_2 a_3 a_4 + 2 a_3^3)^2 \, (- a_1^2 a_4^2 + 6 a_1 a_2 a_3 a_4 - 4 a_1 a_3^3 - 4 a_2^3 a_4 + 3 a_2^2 a_3^2)^3$. Viewing the cubed term as a quadratic in $a_1$, with discriminant $16 A_3^3$, one immediately sees that if $A_3$ is negative, then the discriminant of $\Delta_4$, which is a cubic in $a_0$, is negative for all $a_1$ (except the value of $a_1$ for which the squared term is zero). Hence, the discriminant $\Delta_4$ of the quartic, viewed as a cubic in the free term $a_0$, will have one real root only and will change its sign only once with the variation of the free term $a_0$. The quartic will have only one stationary point, as the derivative of the quartic is quadratic in its free term $a_1$ and the discriminant of this quadratic is proportional to $A_3^3$ and hence negative. Therefore, the quartic will either have no real roots (positive discriminant) or it will have 2 real roots and a pair of complex conjugate roots (negative discriminant).

If, on the other hand, $A_3$ is positive and $a_1$ is such that the squared term is not zero, then the above discriminant of the cubic in $a_0$ is positive. There will be three real roots of the cubic in $a_0$ and the discriminant $\Delta_4$ of the quartic will change sign three times. Thus, with the variation of the free term $a_0$, the quartic with $a_4 > 0$ will ``shift through" the following bands with the increase of $a_0$: negative discriminant with 2 real roots and a pair of complex roots; then a positive discriminant with four real roots; then a negative discriminant with two real roots and a pair of complex roots again; and finally --- a positive discriminant with no real roots.

The values of the free term $a_0$ at which the discriminant of the quartic changes sign (once or three times) are roots of cubic polynomials and, for the general cubic polynomial, these cannot be given in a form that would allow analysis.

Thus, unfortunately, for a polynomial of degree 4 or more, it is not possible to determine the number of real roots just by knowing which derivative(s) do not posses real roots only or by knowing how many real roots each of its derivatives has. Only for polynomials of degree 3 or less this can be done. Moreover, with the increase of the degree of the  polynomial, the Abel--Ruffini theorem prevents from explicitly knowing the points at which the discriminant of the polynomial changes sign (by variation of the free term of the polynomial).

\section{Conclusions}
The reason why the modified Newton's rule works and why the possible number of real roots gets reduced in the presence of a falsely positive quadratic element lies in the fact that a cubic sector with a falsely positive quadratic element simply cannot have three real roots, it can only have one. This is not visible if one only considers the quadratic sectors of the polynomial --- for example, the two quadratic sectors of a cubic may have positive discriminants, that is, the two quadratic elements of the cubic may be positive, but this does not necessarily mean that three real roots of the cubic will exist. With a falsely positive quadratic element, at least two (and up to $2m \le n$) units are ``lost" from the maximum number of possible real roots of the polynomial $p(x)$ from what could be determined with the original Newton's rule. However, it is impossible to calculate exactly by what even number the upper bound on the number of possible real roots will be reduced due to the presence of falsely positive quadratic elements --- it will be by at least a pair for each newly-formed negative group of the falsely positive quadratic elements. That is, one can determine the minimum number of such "lost" real roots, but not their exact number, as noted by Newton \cite{newton} in the case of a negative quadratic element on the basis of his analysis of the quadratic sectors of the polynomial.

\section*{Acknowledgements}

\noindent
It is a great pleasure to thank Jacques G\'elinas for his initial suggestion which triggered this work, for the introduction to the topic and for the constant supply of historical and contemporary references.



\end{document}